\documentclass[12pt, reqno]{amsart}

\usepackage{amsfonts,amsmath,amsthm,amssymb}
\usepackage{latexsym}

 \DeclareMathOperator{\ord}{ord}

\DeclareMathOperator{\lead}{lead}
\DeclareMathOperator{\val}{val}

\begin{document}

{\theoremstyle{plain}
  \newtheorem{theorem}{Theorem}[section]
  \newtheorem{corollary}[theorem]{Corollary}
  \newtheorem{proposition}[theorem]{Proposition}
  \newtheorem{lemma}[theorem]{Lemma}
  \newtheorem{question}[theorem]{Question}
  \newtheorem{conjecture}[theorem]{Conjecture}
  \newtheorem{claim}[theorem]{Claim}
}

{\theoremstyle{definition}
  \newtheorem{definition}[theorem]{Definition}
  \newtheorem{remark}[theorem]{Remark}
  \newtheorem{example}[theorem]{Example}
}

\numberwithin{equation}{section}
\def\QQ{\mathbb{Q}}
\def\ZZ{\mathbb{Z}}
\def\NN{\mathbb{N}}
\def\RR{\mathbb{R}}
\def\C{\mathbb{C}}

\def \a {\alpha}
\def \A {\mathcal{A}}
\def \b {\beta}
\def \B {\mathcal{B}} 
\def \d {\delta}
\def \e {\epsilon}
\def \l {\lambda}
\def \f {\phi}
\def \g {\gamma}
\def \G {\Gamma}
\def \H {\mathcal{H}}
\def \p {\pi}
\def \n {\nu}
\def \m {\mu}
\def \s {\sigma}
\def \S {\Sigma}
\def \w {\omega}
\def \F {\Phi}
\def \O {\mathcal{O}}
\def \W {\Omega}
\def \k {\bold{k}}
\def \bp {\bold{p}}
\def \U {\mathcal{U}}
\def \t {\tau}
\def\D {\mathcal{D}}
\def \P {\mathcal{P}}
\def \T {\mathcal{T}}
\def \R {\mathcal{R}}
\def \et {\eta}
\def \th {\theta}
\def \X {\mathcal{X}}

\def \ra {\rightarrow}

\title[Generating sequences in dim 3]{Constructing examples of semigroups of valuations}

\author{Olga Kashcheyeva}
\address{University of Illinois at Chicago, Department of Mathematics,
Statistics and Computer Science, 851 S. Morgan (m/c 249), Chicago,
IL 60607, USA} \email{kolga@uic.edu}

\begin{abstract} We work with rational rank 1 valuations centered in regular local rings. Given an algebraic function field $K$ of transcendence degree 3 over $k$, a regular local ring $R$  with $QF(R)=K$  and a $k$-valuation $\n$ of $K$, we provide an algorithm for constructing  a generating sequences for $\n$ in $R$. We then develop a method for determining a valuation $\n$ on $k(x,y,z)$ through the sequence of defining values. Using the above results we construct examples of valuations centered in $k[x,y,z]_{(x,y,z)}$ and investigate their semigroups of values.

\end{abstract}

\maketitle

\section{Introduction}

This paper is inspired by the following question: given a regular local noetherian domain $R$ and a valuation $\n$ of the field of fractions $QF(R)$ dominating $R$, what semigroups can appear as a value semigroup $\n(R)$. The answer is available when $R$ is of dimension 1 or 2, but little is known for higher dimensional regular local rings. 

The only semigroups which are realized by a valuation on a one dimensional regular
local ring are isomorphic to the semigroup of  natural numbers. The semigroups which are realized by a valuation on a regular local ring of dimension 2 with algebraically closed residue field are completely classified by Spivakovsky in \cite{Spi}. A different proof for power series ring in two variables over $\C$ is given by Favre and Jonsson in \cite{FJ}. In \cite{C&V}, Cutkosky and Vinh give a necessary and sufficient condition for a semigroup $S$ to be the semigroup of a valuation dominating a regular local ring R of dimension 2 with a prescribed residue field extension. In the context of semigroups under the assumption that the rational rank of $\nu$ is 1 the criterion is as follows, see \cite{C&T}, \cite{C-D-K}, Corollary  \ref{positivity1}, and \cite{C&V}. 

{\it Let $S$ be a well ordered subsemigroup of $\QQ_{\ge 0}$  with at most countable system of generators $\{\b_i\}_{i\ge 0}$ such that $\b_0<\b_1<\dots< \b_n<\dots$. For all $i\ge 0$ let $G_i=\sum_{j=0}^i\b_j\ZZ$ and $q_{i+1}=[G_{i+1}:G_i]=\min\{q\in\ZZ_{>0}|q\b_{i+1}\in G_i\}$.
 Then $S$ is the semigroup of a valuation $\n$ dominating a regular local ring $R$ of dimension 2 if and only if $\b_{i+1}>q_i\b_i$ for all $i\ge 1$}

In particular, it follows that an ordered minimal set of generators $\{\b_i\}_{i\ge 0}$ of the value semigroup of a valuation dominating a regular local ring of dimension 2 is sparse as $\b_{i+1}>2\b_i$ for all $i\ge 1$. This  property does not stand true for higher dimensional regular local rings as shown by example in \cite{C-D-K}.
 
When dimension of $R$ is $n$ the classical results, see \cite{ZS}, state that the value semigroup $\n(R)$ is isomorphic to a well ordered set contained in the nonnegative part of $(\R^h, <_{lex})$ and having an ordinal type of at most $\omega^h$. Here,  $\omega$ is the first infinite ordinal and $h$ is the rank of $\n$; $h$ is less than or equal to the rational rank of $\n$, which is less than or equal to $n$.  Additional bound on the growth of rank 1 valuation semigroups is found by Cutkosky in \cite{Cut}. It leads to a construction of a well ordered subsemigroup of $\QQ_{>0}$ of ordinal type $\omega$, which is not a value semigroup of a noetherian local domain. In \cite{C&T2}, Cutkosky and Teissier formulate bounds on the growth of the number of distinct valuation ideals of $R$ corresponding to values lying in certain parts of the value group of $\n$, thus extending to all ranks the bound given for rank 1 valuations in \cite{Cut}. They also provide some surprising examples of semigroups of rank greater than 1 that occur as semigroups of valuations on noetherian domains, see  \cite{C&T2} and \cite{C&T}. In \cite{Mogh2}, Moghaddam constructs a certain class of value semigroups with large rational rank.

In this paper we use the approach of generating sequences of valuations to investigate value semigroups  of valuations centered in 3-dimensional regular local rings. Let $(R, m_R)$ be a local ring and $K$ be its field of fractions. Let $\n$ be a valuation on $K$ with valuation ring $(V, m_V)$. Assume that $R\subset V$ and $m_R=m_V\cap R$. Let $\F_{R} = \nu(R \backslash \{0\})$ be the semigroup
consisting of the values of nonzero elements of $R$. For $\g \in
\F_R$, let $I_{\g}=\{f\in R \mid \ \n(f)\geq \gamma\}$ and $I^+_{\g}=\{f\in R \mid \ \n(f)> \gamma\}$. A
(possibly infinite) sequence $\{ Q_i \}$ of elements of $R$ is a
{\it generating sequence} of $\nu$  if for every $\g
\in \F_R$ the ideal $I_{\g}$ is generated by the set
$$\{\prod_{i}{Q_i}^{b_i}\mid \ b_i\in \ZZ_{\ge 0},\ 
\sum_{i} b_i \n(Q_i)\geq \gamma\}.$$
Notice that the set of values $\{\n(Q_i)\}$ generates $\F_R$ as a semigroup. Moreover,  the set of images of  $Q_i$ in the associated graded ring of valuation ${\text gr}_{\n}R=\bigoplus_{\g\in\F}I_{\g}/I^+_{\g}$ generate ${\text gr}_{\n}R$ as  $R/m_{R}$-algebra. The graded ring  ${\text gr}_{\n}R$ is of particular interest as it is a key tool  used by Teissier in \cite{T} and \cite{T2}  to solve the local uniformization problem. When the valuation is rational, that is $V/m_V=R/m_R$, the graded ring ${\text gr}_{\n}R$ is isomorphic to the semigroup algebra over $R/m_R$ of the value semigroup $\F_R$, it can be represented as the quotient of a polynomial algebra by a prime binomial ideal, (see \cite{T2} ).

In section \ref{construction} we provide an algorithm for constructing generating sequences of rational rank 1 valuations when $K$ is an algebraic function field of transcendence degree 3 over an algebraically closed field $k$ and $R$ is a regular local ring containing $k$. In the construction we denote the sequence  $\{P_i\}_{i\ge 0}\cup\{T_i\}_{i>0}$ and call it the sequence of jumping polynomials. We then show that $\{P_i\}_{i\ge 0}\cup\{T_i\}_{i>0}$ is a  generating sequence  of valuation in section \ref{properties1}. This construction extends the construction of generating sequences in two dimensional regular local rings used in \cite{GhK}. 

The algorithm is recursive and explicit equations for $P_{i+1}$  in terms of $\{P_j\}_{0\le j\le i}$ and for $T_{\bar{d}(i)}$ in terms of $\{P_j\}_{0\le j\le m_i}\cup\{T_j\}_{0<j\le i}$ are provided. These equations are binomial in nature with the value of the term on the left strictly greater than the value of each term on the right
\begin{align*}
P_{i+1} =P_i^{q_i}-\l_i\prod_{j=0}^{i-1}P_j^{n_{i,j}}\quad\quad\quad\quad\quad\quad\\
T_{\bar{d}} =
T_i^{c_is_i}\prod_{j=0}^{m_i}P_j^{a_j}\prod_{j=0}^{i-1}T_j^{c_j}-\m_{\bar{d}}
\prod_{j=0}^{m_i}P_j^{n_{\bar{d},j}}\prod_{j=0}^{i-1}T_j^{l_{\bar{d},j}}
\end{align*}
Here, $\bar{d}$ is an integer greater than or equal to $i+1$, $\l_i,\m_{\bar{d}}\in k\setminus\{0\}$, $n_{i,j}$,  $m_i$, $a_j$, $c_j$, $n_{\bar{d},j}$, $l_{\bar{d},j}$ are nonnegative integers and $q_i$, $c_i$, $s_i$ are positive integers determined by the algorithm. In the given set up the associated graded ring of valuation ${\text gr}_{\n}R$ is the quotient of a polynomial algebra in infinitely many variables $k[\{P_i\}_i,\{T_i\}_i]$ by the binomial ideal $(\{P_i^{q_i}-\l_i\prod_{j=0}^{i-1}P_j^{n_{i,j}}\}_i,\{T_i^{c_is_i}\prod_{j=0}^{m_i}P_j^{a_j}\prod_{j=0}^{i-1}T_j^{c_j}-\m_{\bar{d}}
\prod_{j=0}^{m_i}P_j^{n_{\bar{d},j}}\prod_{j=0}^{i-1}T_j^{l_{\bar{d},j}}\}_{\bar{d}})$. 

In order to construct examples of semigroups of valuations we work with polynomial rings in three variables $k[x,y,z]$ over an arbitrary base field $k$. We use the approach of extending the trivial valuation of $k$ to a valuation of $k(x,y,z)$ through the sequence of augmented valuations determined by a sequence of defining polynomials as we call them in the construction of section \ref{numerical data}. The technique of sequences of augmented valuations and  key polynomials was first introduced by MacLane in \cite{MacL} in order to describe all possible extensions of a discrete rank one valuation $\m$ of a field $L$ to the field $L(\xi)$.  In \cite{Vaq}, Vaqui$\acute{\text{e}}$ generalized MacLane's axiomatic method to produce all extensions of an arbitrary valuation of a field $L$ to a pseudo-valuation of $L(\xi)$. A different, more constructive,  approach to  describe and generalize key polynomials of MacLane was taken by Herrera Govantes, Olalla Acosta, Mahboub and Spivakovsky in \cite{Sp+} and \cite{Sp+2}, see also \cite{Mah}. The construction of section \ref{numerical data} is most closely related to the construction of key polynomials used in \cite{FJ} by Favre  and Jonsson in order to describe $\C$-valuations on $\C(x,y)$. We note that we do not apply the terminology of key polynomials and augmented valuations when working with defining polynomials.

 The sequence of defining polynomials $\{P_i\}_{i\ge 0}\cup\{Q_i\}_{i>0}$ constructed in section \ref{numerical data} is contained in the ring $k[x,y,z]$. These polynomials are completely determined by the following numerical input:
\begin{itemize}
\item[--] sequence of positive rational numbers $\{\b_i\}_{i\ge 0}$ such that $\b_{i+1}>q_i\b_i$ 
\item[--] sequence of positive rational numbers $\{\bar{\g}_i\}_{i>0}$ such that $\bar{\g}_{i+1}>\bar{r}_{i,0}\b_0+\bar{s}_i\bar{\g}_i$ 
\item[--] sequences of nonzero scalars $\{\l_i\}_{i>0}$ and ${\{\bar{\m}_i\}}_{i>0}$ in $k$
\end{itemize}
Here, $\b_i$ is the prescribed value for $P_i$, $\bar{\g}_i$ is the prescribed value for $Q_i$, $q_i=\min\{q\in\ZZ_{>0}|q\b_i\in\sum_{j=0}^{i-1}\b_j\ZZ\}$, $\bar{s}_i=\min\{s\in\ZZ_{>0}|s\bar{\g}_i\in(\sum_{j=0}^{\infty}\b_j\ZZ+\sum_{j=1}^{i-1}\bar{\g}_j\ZZ)\}$ and $\bar{r}_{i,0}$ is a nonnegative integer  described in the construction of section \ref{numerical data}. Explicit recursive equations for $P_{i+1}$  in terms of $\{P_j\}_{0\le j\le i}$ and for $Q_{i+1}$ in terms of $\{P_j\}_{j\ge 0}\cup\{Q_j\}_{0<j\le i}$ are provided. They are binomial equations
\begin{align*} 
&P_{i+1} =P_i^{q_i}-\l_i\prod_{j=0}^{i-1}P_j^{n_{i,j}}\\
&Q_{i+1} =x^{\bar{r}_{i,0}}Q_i^{\bar{s}_i}-\bar{\m}_i\prod_{j=0}^{\bar{m}_i}
P_j^{\bar{n}_{i,j}}\prod_{j=1}^{i-1}Q_j^{\bar{l}_{i,j}}
\end{align*}
Here, $\bar{m}_i$ and $n_{i,j},\bar{n}_{i,j},\bar{l}_{i,j}$ are nonnegative integers determined by the construction of section \ref{numerical data}.

It is shown in section \ref{properties2} that provided infinitely many $q_i$ and $\bar{s}_i$ are greater than 1 the numerical data above uniquely determines a valuation on $k(x,y,z)$. In particular, $(P_0,\b_0)$ determines a discrete valuation on $k(x)$. Polynomials $\{(P_i, \b_i)\}_{i>0}$ determine the extension of the discrete valuation of $k(x)$ to $k(x,y)$ and $\{(Q_i,\bar{\g}_i)\}_{i>0}$ determine the extension of the valuation of $k(x,y)$ to $k(x,y,z)$. Polynomials $\{P_i\}_{i>0}$ are monic polynomials in $k(x)[y]$, see Proposition \ref{P-degrees}. Polynomials $\{Q_i,\}_{i>0}$ are not in general monic polynomials in $k(x,y)[z]$, see Proposition \ref{Q-degrees-intermediate} and Corollary \ref{Q-degrees}.  We note that in MacLane's construction key polynomials are monic polynomials in $L[\xi]$, however there is no restriction on the lower degree terms of key polynomials except that the coefficients are elements of $L$. In our construction coefficients of $P_i$ are elements of the ring $k[x]$ and coefficients of $Q_i$ are elements of the ring $k[x,y]$.

 In sections \ref{main_example} and \ref{examples}  we provide examples of semigroups of valuations centered in $k[x,y,z]_{(x,y,z)}$. We use defining polynomials to construct a valuation and then work with the sequence of jumping polynomials to describe its value semigroup. One of the examples shows that when the set $\{\bar{r}_{i,0}|\bar{r}_{i,0}>0\}$ is empty the generators of the value semigroup are the values of defining polynomials $\{\b_i\}_{i\ge 0}\cup\{\bar{\g}\}_{i>0}$. In our main example, section \ref{main_example},  only one member of the sequence $\{\bar{r}_{i,0}\}_{i\ge 0}$ is greater than 0.
Finally, in the last example we set $\bar{r}_{1,0}$  and $\bar{r}_{2,0}$ greater than 0. Already in the case of just two $\bar{r}_{i,0}$ greater than zero the pattern for the sequence of generators of the value semigroup becomes quite complicated. 


\section{Construction of jumping polynomials}\label{construction}

We assume that $k$ is an algebraically closed field and $K$ is an
algebraic function field of transcendence degree 3 over $k$.
$\n$ is a $k$-valuation of $K$ with valuation ring $(V,m_V)$ and
value group $\G$. We assume that $\n$ is of rational rank 1 and
dimension 0, so  that $\G$ is a subgroup of $\QQ$ and $V/m_V=k$.
$(R, m_R)$ is a local subring of $K$ with $k\subset R$
and $R_{(0)}=K$. We assume that $R$ is a regular ring with regular parameters $x,y,z$. 
We also assume that $R$ is dominated by $\n$, that is $R\subset V$ and $R\cap m_V=m_R$.  

We use the following notations. Let $a_1,a_2,\dots,a_k$ be nonnegative integers and $a$ be a positive integer. If
$M=x_1^{a_1}x_2^{a_2}\cdots x_k^{a_k}$ is  a monomial in
$x_1,x_2,\dots,x_k$ and $\X$ is a set of monomials in (infinitely many) variables $x_1,x_2,\dots,x_k,x_{k+1},\dots$ then $M$ is said to be irreducible with respect to $\X$ if $x_1^{b_1}x_2^{b_2}\cdots x_k^{b_k}\notin\X$ for all
$b_1, b_2,\dots, b_k\in\ZZ$ such that $0\le b_i\le a_i$. If $\F$ is a semigroup with a fixed set of generators $\{\m_1,\m_2,\dots,\m_k\}$ and
$\m\in(\F+(-\F))$ we say that $(a_1,a_2,\dots,a_k,a)$ is reduced with
respect to $(\F,\{\m_1,\dots,\m_k\},\m)$, or $(\F,\m)$ for short, if  $a_1\m_1+a_2\m_2+\dots+a_k\m_k+a\m\in\F$ and
$b_1\m_1+b_2\m_2+\dots+b_k\m_k+b\m\notin\F$ for all
$b_1, b_2,\dots, b_k,b\in\ZZ$ such that $0<b\le a$, $0\le b_i\le a_i$ and
$(\sum_{i=1}^k b_i+b)<(\sum_{i=1}^k a_i+a)$.

To construct a generating sequence of $\n$ in $R$ we define
a sequence of jumping polynomials $\{P_i\}_{i\ge 0}\cup\{T_i\}_{i>0}$ in $R$.  

Let $P_0= x,\; P_1=y$ and  $\P_0=\P_1=\emptyset$. For all $i\ge 0$
we set $\b_i=\n(P_i)$ and
$$
\begin{array}{rl}
G_i&=\sum_{j=0}^{i}\b_j\ZZ \\
S_i&=\sum_{j=0}^{i}\b_j\ZZ_{\ge 0} .\\
\end{array}$$
For all $i>0$ let $q_i=\min\{q\in\ZZ_{>0}\,|\,q\b_i\in G_{i-1}\}$.
Set $\P_{i+1}=\P_i\cup\{P_i^{q_i}\}$. Let
$n_{i,0},n_{i,1},\dots,n_{i,i-1}\in\ZZ_{\ge 0}$ be such that
$q_i\b_i=\sum_{j=0}^{i-1}n_{i,j}\b_j$ and $\prod_{j=0}^{i-1}P_j^{n_{i,j}}$
is irreducible with respect to $\P_i$. Denote by $\l_i$ the residue of
$P_i^{q_i}/\prod_{j=0}^{i-1}P_j^{n_{i,j}}$ in $V/m_V$ and set
$$
P_{i+1}=P_i^{q_i}-\l_i\prod_{j=0}^{i-1}P_j^{n_{i,j}}.
$$
Finally set $G=\bigcup_{i=0}^\infty G_i$,  $S=\bigcup_{i=0}^\infty S_i$
and $\P=\bigcup_{i=0}^\infty\P_i$.

\begin{remark}
The infinite sequence $\{P_i\}_{i\ge 0}$ is well defined due to the
following

\begin{itemize}
\item[1.] $P_i\neq 0$ and $\b_i<\infty$ for all $i>0$. (See Proposition \ref{P-degrees1})

\item[2.] $\b_{i+1}>q_i\b_i$ for all $i>0$.

\item[3.] $\prod_{j=0}^{i-1}P_j^{n_{i,j}}$ is irreducible with respect to $\P_i$
if and only if $n_{i,j}<q_j$ for all $j>0$.

\item[4.] $n_{i,0},n_{i,1},\dots,n_{i,i-1}\in\ZZ_{\ge 0}$ is a unique $i$-tuple
of integers satisfying the conditions $q_i\b_i=\sum_{j=0}^{i-1}n_{i,j}\b_j$
and $n_{i,j}<q_j$ for all $j>0$.
(See Lemma \ref{group representation} and Corollary \ref{positivity1}.)

\item[5.] $\l_i\in k\backslash\{0\}$ for all $i>0$.
\end{itemize}
\end{remark}

Let $T_1=z$, $m_0=0$, $\d_0=1$, $\T_1=\P$, $H_0=\{0\}$ and
$U_0=\{0\}$. For all $i>0$ such that $T_i\neq 0$ we set  $\g_i=\n(T_i)$ and
$$
\begin{array}{rl}
H_i&=\sum_{j=1}^{i}\g_j\ZZ \\
U_i&=\sum_{j=1}^{i}\g_j\ZZ_{\ge 0} .\\
\end{array}
$$
For all $i>0$ such that $T_i\neq 0$ let
$$
\begin{array}{rl}
s_i &=\min\{s\in\ZZ_{>0}\,\mid s\g_i\in (H_{i-1}+G)\}\\
m_i &=\max(m_{i-1}, \min\{j\in \ZZ_{\ge 0}\,\mid s_i\g_i\in(H_{i-1}+G_j)\}).\\
\end{array}
$$
For all $i>0$ such that $T_i\neq 0$ consider the following sets of $(m_i+1+i)$-tuples
of nonnegative integers
$$\begin{array}{rl}
\D_i =\{&\bar{d}=(a_0,a_1,\dots,a_{m_i},c_1,\dots,c_{i-1},c_i)\,\mid
\bar{d}\in(\ZZ_{\ge 0})^{m_i+1+i}, c_i>0,\\
 & \bar{d} \text{ is reduced with
respect to } ((S_{m_i}+U_{i-1}),s_i\g_i),\\
&\prod_{j=0}^{m_i}P_j^{a_j} \prod_{j=1}^{i-1}T_j^{c_j}
\text{ is irreducible with respect to }  \T_i\}.
\end{array}
$$
Set $\d_i=\#\D_i$ and $\T_{i+1}=\T_i\cup\{T_i^{c_is_i}\prod_{j=0}^{m_i}P_j^{a_j}
\prod_{j=1}^{i-1}T_j^{c_j}\mid \bar{d}\in\D_i\}$. If $\bar{d}\in\D_i$ denote
$|\bar{d}|=\sum_{j=0}^{m_{i}}a_j\b_j+\sum_{j=0}^{i-1}c_j\g_j+c_is_i\g_i$. We
assume that $\D_i=\{\bar{d}_1,\bar{d}_2,\dots,\bar{d}_{\d_i}\}$ is an ordered set
with $|\bar{d}_1|\le|\bar{d}_2|\le\dots\le|\bar{d}_{\d_i}|$. If
$\bar{d}=\bar{d}_k$ for some $1\le k\le\d_i$  let
$n_{\bar{d},0},n_{\bar{d},1},\dots,n_{\bar{d},m_i},
l_{\bar{d},1},l_{\bar{d},2},\dots,l_{\bar{d},i-1}\in\ZZ_{\ge 0}$ be such that
$|\bar{d}|=\sum_{j=0}^{m_i}n_{\bar{d},j}\b_j+\sum_{j=1}^{i-1}l_{\bar{d},j}\g_j$
and $\prod_{j=0}^{m_i}P_j^{n_{\bar{d},j}} \prod_{j=1}^{i-1}T_j^{l_{\bar{d},j}} $
is irreducible with respect to $\T_i$. Denote by $\m_{\bar{d}}$ the
residue of $(T_i^{c_is_i}\prod_{j=0}^{m_i}P_j^{a_j}\prod_{j=0}^{i-1}T_j^{c_j})/
(\prod_{j=0}^{m_i}P_j^{n_{\bar{d},j}}\prod_{j=0}^{i-1}T_j^{l_{\bar{d},j}})$
in $V/m_V$ and set
$$
T_{(\sum_{j=0}^{i-1}\d_j)+k}=T_{\bar{d}}=
T_i^{c_is_i}\prod_{j=0}^{m_i}P_j^{a_j}\prod_{j=0}^{i-1}T_j^{c_j}-\m_{\bar{d}}
\prod_{j=0}^{m_i}P_j^{n_{\bar{d},j}}\prod_{j=0}^{i-1}T_j^{l_{\bar{d},j}}.
$$
For all $i>0$ such that $T_i=0$ set $\g_i=0$, $H_i=H_{i-1}$, $U_i=U_{i-1}$, $s_i=1$, $m_i=m_{i-1}$, $\D_i=\emptyset$, $\d_i=0$ and $\T_{i+1}=\T_i\cup\{T_i\}$. Finally set $H=\bigcup_{i=0}^\infty H_i$,  $U=\bigcup_{i=0}^\infty U_i$
and $\T=\bigcup_{i=0}^\infty\T_i$.

\begin{remark}
The sequence $\{T_i\}_{i\ge 1}$ is well defined due to the
following
\begin{itemize}
\item[1.] $\d_i<\infty$ for all $i>0$. (See Lemma \ref{set-bound}.)

\item[2.] If $i,k>0$ then $\prod_{j=0}^kP_j^{a_j} \prod_{j=1}^{i-1}T_j^{c_j}$ is irreducible with respect to $\T$ if and only if it is irreducible with respect to $\T_i$.

\item[3.] If $\bar{d}\in \D_i$ for some $i>0$ then $n_{\bar{d},0},n_{\bar{d},1},
\dots,n_{\bar{d},m_i},l_{\bar{d},1},l_{\bar{d},2},\dots,l_{\bar{d},i-1}\in\ZZ_{\ge 0}$
is a unique $(m_i+i)$-tuple of integers satisfying the conditions
$|\bar{d}|=\sum_{j=0}^{m_i}n_{\bar{d},j}\b_j+\sum_{j=1}^{i-1}l_{\bar{d},j}\g_j$
and $\prod_{j=0}^{m_i}P_j^{n_{\bar{d},j}} \prod_{j=1}^{i-1}T_j^{l_{\bar{d},j}} $
is irreducible with respect to $\T$.
(See Proposition \ref{uniqueness}.)

\item[4.] $\m_{\bar{d}}\in k\backslash\{0\}$ for all $\bar{d}\in\D_i$ and all $i>0$.

\item[5.] It may happen that $T_{\bar{d}}=0$ for some $i$ and $\bar{d}\in\D_i$. (See example \ref{Example2})
\end{itemize}
\end{remark}

For $i,j>0$ we say that $T_j$ is an immediate successor of $T_i$ and $T_i$ is an immediate predecessor of $T_j$ if $T_j=T_{\bar{d}}$ for some $\bar{d}\in\D_i$. We say that $T_j$ is a successor of $T_i$ if there exists  $k\in\ZZ_{>0}$ and a sequence of integers $l_0,l_1,\dots,l_k$ such that $l_0=i$, $l_k=j$ and  $T_{l_t}$ is an immediate successor of $T_{l_{t-1}}$ for all $1\le t\le k$. Notice that if $T_k\neq 0$ is an immediate successor of $T_i$ then $\g_k> s_i\g_i$.

We say that $T_i$ is redundant if $T_i=0$ or the following conditions are satisfied: $\n(T_i)\in (S_{m_i}+U_{i-1})$, $T_i$ has a finite number of nonzero successors and the only immediate successor $T_{\sum_{j=0}^i\d_j}$ of $T_i$ is redundant.



\section{Preliminary results}\label{preliminaries}

In this section we prove several arithmetical statements needed to justify the construction of  jumping polynomials in section \ref{construction} as well as construction of defining polynomials in section \ref{numerical data}. So we assume that $\{\b'_i\}_{i\ge 0}$ and $\{\g'_i\}_{i>0}$ are sequences of rational numbers but  not necessarily values of jumping polynomials. As above we set $m'_0=0$, $G'_0=\b'_0\ZZ$, $H'_0=\{0\}$ and for all $i> 0$ we set 
\begin{align*}
S'_i=\sum_{j=0}^i\b'_j\ZZ_{\ge 0},\quad G'_i=\sum_{j=0}^i \b'_j\ZZ\quad&\text{ and }\quad S'=\bigcup_{j=0}^{\infty}S'_j,\quad G'=\bigcup_{j=0}^{infty}G'_j\\
U'_i=\sum_{j=1}^i \g'_j\ZZ_{\ge 0},\quad H'_i=\sum_{j=1}^i \g'_j\ZZ\quad &\text{ and }\quad U'=\bigcup_{j=1}^{\infty}U'_j,\quad H'=\bigcup_{j=1}^{\infty}H'_j\\
q'_i=\min\{q\in\ZZ_{>0}\mid q\b'_i\in G'_{i-1}\},\quad\quad &\; s'_i=\min\{s\in\ZZ_{>0}\mid s\g'_i\in(H'_{i-1}+G')\}\\
 \text {and }\quad m'_i =\max(m'_{i-1}, \min & \{j\in\ZZ_{\ge 0}\mid  s'_i\g'_i\in(H'_{i-1}+G'_j)\})
\end{align*}  

We notice that if $i>0$ and $k\ge m'_i$ then $G'_i=\frac{1}{q'_i}G'_{i-1}$ and $(G'_k+H'_i)=\frac{1}{s'_i}(G'_k+H'_{i-1})$ so that $(G'_k+H'_i)=(\prod_{j=1}^i\frac{1}{s'_j})G'_k=(\prod_{j=1}^k\frac{1}{q'_j}\prod_{j=1}^{i}\frac{1}{s'_j})G'_0$. Therefore, if $i\ge0$ and $k>m'_i$ then $(G'_k+H'_i)=\frac{1}{q'_k}(G'_{k-1}+H'_i)$.

\begin{lemma}\label{group representation}
Suppose that $\a\in(G'_k+H'_i)$ for some $i,k\ge 0$. Then there exist $n_0\in\ZZ$ and  a unique $(m+i)$-tuple of nonnegative integers $(n_1,\dots,n_m,l_1,\dots,l_i)$
such that  $m=max(k,m'_i)$, $\a=\sum_{j=0}^m n_j\b'_j+\sum_{j=1}^i l_j\g'_j$ 
and $n_j<q'_j$, $l_j<s'_j$ for all $j>0$.
\end{lemma}

\begin{proof}
We first observe that if $i$ is fixed and $k\le m'_i$ then $\a\in(G'_k+H'_i)$ implies
$\a\in(G'_{m_i}+H'_i)$ whereas the conclusion of the lemma does not depend on $k$.
Thus, it is enough to show that the statement holds for all $i\ge 0$ and $k\ge m'_i$.

We use induction on $(i,k)$. If $i=0$ and $k=m'_0=0$ then the statement is trivial.
Assume that $i$ is fixed and $k>m'_i$. Then since $(G'_k+H'_i)=\frac{1}{q'_k}(G'_{k-1}+H'_i)$ and $\b'_k$ generates $(G'_k+H'_i)$ over $(G'_{k-1}+H'_i)$ we can write
$$
(G'_k+H'_i)=(G'_{k-1}+H'_i)\cup(\b'_k+G'_{k-1}+H'_i)\cup\dots\cup((q'_k-1)\b'_k+G'_{k-1}+H'_i),
$$
where the union on the right is a disjoint union of $(G'_{k-1}+H'_i)$-sets. There
exists a unique nonnegative integer $n_k<q'_k$ such that $\a\in(n_k\b'_k+G'_{k-1}+H'_i)$. Using the inductive hypothesis for $(\a-n_k\b'_k)\in(G'_{k-1}+H'_i)$ we find  $n_0\in\ZZ$ and a $(k-1+i)$-tuple of nonnegative integers $(n_1,\dots,n_{k-1},l_1,\dots,l_i)$ such that $\a=\sum_{j=0}^k n_j\b'_j+\sum_{j=1}^i l_j\g'_j$ and $n_j<q'_j$, $l_j<s'_j$for all $j>0$.

To check that such representation is unique assume that $\a=\sum_{j=0}^k n'_j\b'_j+\sum_{j=1}^i l'_j\g'_j$ is another representation satisfying the requirements of the lemma. Then $0\le n'_k<q'_k$ and $(\a-n'_k\b'_k)\in(G'_{k-1}+H'_i)$.
On the other hand $n_k$ is the unique integer such that $0\le n_k<q'_k$ and
$\a\in(n_k\b'_k+G'_{k-1}+H'_i)$, thus, $n'_k=n_k$. Then by the inductive
hypothesis applied to $\sum_{j=0}^{k-1}n_j\b'_j+\sum_{j=1}^i l_j\g'_j=\sum_{j=0}^{k-1}n'_j\b'_j+\sum_{j=1}^i l'_j\g'_j$ we have $n'_0=n_0$ and  $n'_j=n_j$, $l'_j=l_j$ for all $j>0$.

Assume now that $i>0$ and $k=m'_i$. Then  since $(G'_k+H'_i)=\frac{1}{s'_i}(G'_k+H'_{i-1})$ and $\g'_i$ generates $(G'_k+H'_i)$ over $(G'_k+H'_{i-1})$ we can write
$$
(G'_k+H'_i)=(G'_k+H'_{i-1})\cup(\g'_i+G'_k+H'_{i-1})\cup\dots\cup((s'_i-1)\g'_i+G'_k+H'_{i-1}),
$$
where the union on the right is a disjoint union of $(G'_k+H'_{i-1})$-sets. Let $l_i<s'_i$
be a nonnegative integer such that $\a\in(l_i\g'_i+G'_k+H'_{i-1})$. Using the inductive
hypothesis for $(\a-l_i\g'_i)\in(G'_k+H'_{i-1})$ we find  $n_0\in\ZZ$
and a $(k+i-1)$-tuple of nonnegative integers $(n_1,\dots,n_k,l_1,\dots,l_{i-1})$
such that $\a=\sum_{j=0}^k n_j\b'_j+\sum_{j=1}^i l_j\g'_j$ and $n_j<q'_j$, $l_j<s'_j$
for all $j>0$. We deduce that such representation is unique using an argument similar to the one above.
\end{proof}

The next statement is a straight forward consequence of Lemma \ref{group representation}. 

\begin{corollary}\label{order1}
Suppose that $(a_0,\dots,a_k,c_1,\dots,c_i)$, $(b_0,\dots,b_k,d_1,\dots,d_i)$ are $(k+1+i)$-tuples of nonnegative integers such that $a_j,b_j<q'_j$ and $c_j,d_j<s'_j$ for all $j>0$. 

If $(a_0,\dots,a_k,c_1,\dots,c_i)\neq(b_0,\dots,b_k,d_1,\dots,d_i)$ then $\sum_{j=0}^k a_j\b'_j+\sum_{j=1}^ic_j\g'_j\neq\sum_{j=0}^k b_j\b'_j+\sum_{j=1}^i d_j\g'_j$.
\end{corollary}

We find a sufficient condition for positivity of $n_0$ by putting additional assumptions on the sequence of rational numbers $\{\b'_i\}_{i\ge 0}$.

\begin{corollary}\label{positivity1}
Suppose that the sequence $\{\b'_{i}\}_{i\ge 0}$ satisfies $\b'_{i+1}>q'_i\b'_i$ for all $i>0$ and $\b'_0>0$. Suppose that $k>0$ and $\a\in G'_k$ is such that $\a\ge q'_k\b'_k$. 

Let $\a=\sum_{j=0}^kn_j\b'_j$ be the  representation given by  Lemma \ref{group representation} then $n_0>0$.
\end{corollary}

\begin{proof} 
Since $q'_j\b'_j<\b_{j+1}$ and $(n_j+1)\le q'_j$ for all $j>0$ we have
$$
\sum_{j=1}^k n_j\b'_j<q'_1\b'_1+\sum_{j=2}^k n_j\b'_j<q'_2\b_2+\sum_{j=3}^k n_j\b'_j
<\dots<q'_{k-1}\b'_{k-1}+n_k\b'_k<q'_k\b'_k.
$$
If $\a\ge q'_k\b'_k$ this implies that $n_0\b'_0\ge q'_k\b'_k-\sum_{j=1}^k n_j\b'_j>0$. Thus $n_0>0.$
\end{proof}


\section{Properties of jumping polynomials}\label{properties1}

In this section we show that the sequence of jumping polynomials
$\{P_i\}_{i\ge 0}\cup\{T_i\}_{i>0}$ as constructed above is well defined
and is a generating sequence of $\n$ in $R$. 

We use notation of section \ref{construction}. If $A$ is a commutative ring and $h\in A$ we define $h^0=1$, in particular,  $h^0=1$ when $h=0$. If $f\in A[y]\setminus\{0\}$ then $\deg_y f$ and $\lead_y f$ denote the degree and leading coefficient  of $f$ as a polynomial in $y$ with coefficients in $A$. 

\begin{proposition}\label{P-degrees1} Suppose that $A=k[x]$ and $i\ge 1$. Then $P_i$ is a monic polynomial in $A[y]$ and $\deg_y P_i=\prod_{j=1}^{i-1}q_j$. In particular, $P_i\neq 0$ for all $i>0$.
\end{proposition}

\begin{proof} 
We use induction on $i$. For $i=1$ we have $P_1=y$ is monic of degree 1. Let $i>1$ then 
$P_i=P_{i-1}^{q_{i-1}}-\l_{i-1}\prod_{j=0}^{i-2}P_j^{n_{i,j}}$. By the inductive hypothesis we have  
\[
\begin{split}
\deg_y P_{i-1}^{q_{i-1}} =\prod_{j=1}^{i-1} &q_j,\\
\deg_y (\prod_{j=0}^{i-2}P_j^{n_{i-1,j}}) & =n_{i-1,1}+n_{i-1,2}q_1+\dots+n_{i-1,i-2}q_1\cdots q_{i-3}\\
& <q_1+(q_2-1)q_1+\dots+(q_{i-2}-1)q_1\cdots q_{i-3}=q_1\cdots q_{i-2}.
\end{split}
\]
Thus $\deg_y P_i=\prod_{j=1}^{i-1} q_j$ and $\lead_y P_i=\lead_y P_{i-1}^{q_{i-1}}=1$.
\end{proof}

The next two statements show that the subsequence of jumping polynomials $\{T_i\}_{i>0}$ is well defined.

\begin{lemma}\label{set-bound}
If $i>0$ then $\d_i<\infty$.
\end{lemma}

\begin{proof}
Fix $j>0$ and set
$$
\begin{array}{rl}
F_{j,0} &=\{r\in\ZZ_{\ge 0} \mid (r\b_0+s_j\g_j)\in(S_{m_j}+U_{j-1})\}\\
F_j &=\{r\in\ZZ_{>0} \mid rs_j\g_j\in(S_{m_j}+U_{j-1})\}\\
\end{array}
$$
$F_j$ and $F_{j,0}$ are nonempty sets since
$(G_{m_j}+H_{j-1})\cap(N,\infty)=(S_{m_j}+U_{j-1})\cap(N,\infty)$ when
$N\in\ZZ$ is big enough. Let $r_{j,0}=\min(F_{j,0})$ and $r_j=\min(F_j)$
then $(r_{j,0},0,\dots,0,1),(0,\dots,0,r_j)\in\D_j$ and therefore
$x^{r_{j,0}}T_j^{s_j},T_j^{r_js_j}\in\T_i$ for all $i>j$.

Let $i>0$ and $\bar{d}=(a_0,a_1,\dots,a_{m_i},c_1,\dots,c_{i-1},c_i)\in\D_i$.
Since $\prod_{j=0}^{m_i}P_j^{a_j}\prod_{j=0}^{i-1}T_j^{c_j}$ is irreducible
with respect to $\T_i$ we get that $a_j<q_j$ for all $j>0$ and $c_j<r_js_j$
for all $j<i$. Also $a_0\le r_{i,0}$ and $c_i\le r_i$ since $\bar{d}$ is reduced
with respect to $((S_{m_i}+U_{i-1}),s_i\g_i)$. Thus
$\d_i\le (r_{i,0}+1)q_1\cdots q_{m_i}r_1s_1\cdots r_{i-1}s_{i-1}r_i$.
\end{proof}

From now on for all $i>0$ we fix the notation for $r_i$ as introduced in the proof of
Proposition \ref{set-bound}
$$
r_i=\min\{r\in\ZZ_{>0} \mid rs_i\g_i\in(S_{m_i}+U_{i-1})\}.
$$

\begin{proposition}\label{uniqueness}
Suppose that $\a\in(S_k+U_i)$ for some $k,i\ge 0$. Then there exists a unique
\mbox{$(m+1+i)$}-tuple of nonnegative integers
$(n_0,\dots, n_m, l_1,\dots,l_i)$ such that $m=\max(m_i,k)$,
$\a=\sum_{j=0}^m n_j\b_j+\sum_{j=1}^i l_j\g_j$ and
$\prod_{j=0}^m P_j^{n_j}\prod_{j=1}^i T_j^{l_j}$ is irreducible with respect
to $\T$.
\end{proposition}

\begin{proof}
We first observe that if $i$ is fixed and $k\le m_i$ then $\a\in(S_k+U_i)$ implies
$\a\in(S_{m_i}+U_i)$ whereas the conclusion of the lemma does not depend on $k$.
Thus, it is enough to show that the statement holds for all $i\ge 0$ and $k\ge m_i$.
We use induction on $(i,k)$. If $i=0$ and $k=m_0=0$ then the statement is trivial.

Assume that $i$ is fixed and $k>m_i$. Then since $\b_k$ generates $(S_k+U_i)$ over
$(S_{k-1}+U_i)$ and $q_k\b_k\in (S_{k-1}+U_i)$ we can write
\begin{equation}\label{1}
(S_k+U_i)=(S_{k-1}+U_i)\cup(\b_k+S_{k-1}+U_i)\cup\dots\cup((q_k-1)\b_k+S_{k-1}+U_i)
\end{equation}
Moreover, if \mbox{$(a_1\b_k+S_{k-1}+U_i)\cap(a_2\b_k+S_{k-1}+U_i)\neq\emptyset$} for
some nonnegative integers $a_1<a_2<q_k$, then we have $(a_2-a_1)\b_k\in(G_{k-1}+H_i)$ 
with $(a_2-a_1)>0$. Thus $(a_2-a_1)$ is a multiple of $q_k$. This contradicts the 
assumption that $a_2<q_k$. Therefore the union on the right of (\ref{1}) is a disjoint 
union of $(S_{k-1}+U_i)$-sets.

There exists a unique nonnegative integer $n_k<q_k$ such that $\a\in(n_k\b_k+S_{k-1}+U_i)$.
Using the inductive hypothesis for $(\a-n_k\b_k)\in(S_{k-1}+U_i)$ we find a $(k+i)$-tuple
of nonnegative integers $(n_0,\dots, n_{k-1}, l_1,\dots,l_i)$ such that
$\a=\sum_{j=0}^k n_j\b_j+\sum_{j=1}^i l_j\g_j$ and
$\prod_{j=0}^{k-1} P_j^{n_j}\prod_{j=1}^i T_j^{l_j}$ is irreducible with respect
to $\T$. It remains to check that $\prod_{j=0}^k P_j^{n_j}\prod_{j=1}^i T_j^{l_j}$ is
irreducible with respect to $\T$.

Denote $\prod_{j=0}^k P_j^{n_j}\prod_{j=1}^i T_j^{l_j}$ by  $M$ and
$\prod_{j=0}^{k-1} P_j^{n_j}\prod_{j=1}^i T_j^{l_j}$ by $M_1$, also if $t>0$ and 
$\bar{d}\in\D_t$ denote the components of $\bar{d}$ by $(a_0,\dots,a_{m_t},c_1,\dots,c_{t-1},c_t)$. We notice 
that $M$ is irreducible with respect to $\P$ since $M_1$ is irreducible with respect to $\P$
and $n_k<q_k$. Also if $\bar{d}\in\D_t$ for some  $t>i$ then since $M$ does not have $T_t$ 
in it, $M$ is irreducible with respect to the subset
$\{T_t^{c_ts_t}\prod_{j=0}^{m_t} P_j^{a_j}\prod_{j=1}^{t-1} T_j^{c_j}\}$ of $\T_{t+1}$
corresponding to $\bar{d}$. Finally if $\bar{d}\in\D_t$ for some  $t\le i$ then since 
$k>m_i\ge m_t$ and therefore 
$\{T_t^{c_ts_t}\prod_{j=0}^{m_t} P_j^{a_j}\prod_{j=1}^{t-1} T_j^{c_j}\}$ does not have $P_k$
in it, $M$ is reducible with respect to the subset
$\{T_t^{c_ts_t}\prod_{j=0}^{m_t} P_j^{a_j}\prod_{j=1}^{t-1} T_j^{c_j}\}$ of $\T_{t+1}$
corresponding to $\bar{d}$ if and only if $M_1$ is reducible with respect to this
subset. Thus we deduce that $M$ is irreducible with respect to $\T$ as required.

To check that such representation is unique we use the same argument as in the proof of
Lemma \ref{group representation}. Assume that $\a=\sum_{j=0}^k n'_j\b_j+\sum_{j=1}^i l'_j\g_j$
is another representation satisfying the requirements of the proposition. Then $0\le n'_k<q_k$
and $(\a-n'_k\b_k)\in(S_{k-1}+U_i)$. On the other hand $n_k$ is the unique integer such that
$0\le n_k<q_k$ and $\a\in(n_k\b_k+S_{k-1}+U_i)$, thus, $n'_k=n_k$. Then by the inductive
hypothesis since $\a-n_k\b_k=\sum_{j=0}^{k-1}n'_j\b_j+\sum_{j=1}^i l'_j\g_j$ we have
$n'_j=n_j$ and $l'_j=l_j$ for all $j$.

Assume now that $i>0$ and $k=m_i$. Since $\g_i$ generates $(S_k+U_i)$ over
$(S_k+U_{i-1})$ and $r_is_i\g_i\in(S_k+U_{i-1})$ we can write
$$
(S_k+U_i)=(S_k+U_{i-1})\cup(\g_i+S_k+U_{i-1})\cup\dots\cup((r_is_i-1)\g_i+S_k+U_{i-1}),
$$
where the union of $(S_k+U_{i-1})$-sets on the right does not have to be a disjoint union
any more. Let $l_i<r_is_i$ be the minimal nonnegative integer such that
$\a\in(l_i\g_i+S_k+U_{i-1})$. Using the inductive hypothesis for $(\a-l_i\g_i)\in(S_k+U_{i-1})$
we find a $(k+i)$-tuple of nonnegative integers $(n_0,\dots, n_k, l_1,\dots,l_{i-1})$ such that
$\a=\sum_{j=0}^k n_j\b_j+\sum_{j=1}^i l_j\g_j$ and
$\prod_{j=0}^k P_j^{n_j}\prod_{j=1}^{i-1} T_j^{l_j}$ is irreducible with respect to $\T$. It
remains to check that $\prod_{j=0}^k P_j^{n_j}\prod_{j=1}^i T_j^{l_j}$ is irreducible with
respect to $\T$.

Denote $\prod_{j=0}^k P_j^{n_j}\prod_{j=1}^i T_j^{l_j}$ by  $M$ and
$\prod_{j=0}^k P_j^{n_j}\prod_{j=1}^{i-1} T_j^{l_j}$ by $M_1$, also  if $t>0$ and 
$\bar{d}\in\D_t$ denote the components of $\bar{d}$ by $(a_0,\dots,a_{m_t},c_1,\dots,c_{t-1},c_t)$. We notice that 
$M$ is irreducible with respect to $\T_i$ since $M_1$ is irreducible with respect to $\T_i$.
Also if $\bar{d}\in\D_t$ for some  $t>i$ then since $M$ does not have $T_t$ in it,
$M$ is irreducible with respect to the subset
$\{T_t^{c_ts_t}\prod_{j=0}^{m_t} P_j^{a_j}\prod_{j=1}^{t-1} T_j^{c_j}\}$ of $\T_{t+1}$
corresponding to $\bar{d}$. Finally assume that there exists $\bar{d}\in\D_i$ such that
$M$ is reducible with respect to the subset
\mbox{$\{T_i^{c_is_i}\prod_{j=0}^k P_j^{a_j}\prod_{j=1}^{i-1} T_j^{c_j}\}$}
of $\T_{i+1}$ corresponding to $\bar{d}$. Then we have
$$
\a=\sum_{j=0}^k(n_j-a_j)\b_j+\sum_{j=0}^{i-1}(l_j-c_j)\g_j+(l_i-c_is_i)\g_i+|\bar{d}|,
$$
where all the coefficients in the summations are nonnegative integers,
$|\bar{d}|\in(S_k+U_{i-1})$ and $(l_i-c_is_i)<l_i$. Thus $\a\in((l_i-c_is_i)\g_i+S_k+U_{i-1})$,
which contradicts the choice of $l_i$ as the smallest nonnegative integer such that
$\a\in(l_i\g_i+S_k+U_{i-1})$. This shows that $M$ is irreducible with respect to $\T$.

To check that such representation is unique assume that
$\a=\sum_{j=0}^k n'_j\b_j+\sum_{j=1}^i l'_j\g_j$ is another representation satisfying the
requirements of the proposition. If $l_i=l'_i$ then applying the inductive hypothesis to
$\a-l_i\g_i$ we deduce that $n_j=n'_j$ and $l_j=l'_j$ for all $j$. Thus it 
suffices to show that $l_i=l'_i$.

Denote by $K$ the set of indices $j$ such that $n_j> n'_j$ and by $K'$ the set of indices $j$ 
such that $n'_j>n_j$, denote by $I$ the set of indices $j$ such that $l_j>l'_j$ and by $I'$ 
the set of indices $j$ such that $l'_j>l_j$. Then we can write
\begin{equation}\label{*}
\sum_{j\in K}(n_j-n'_j)\b_j+\sum_{j\in I}(l_j-l'_j)\g_j=
\sum_{j\in K'}(n'_j-n_j)\b_j+\sum_{j\in I'}(l'_j-l_j)\g_j,
\end{equation}
where all the coefficients in the summations are positive integers. We will show that 
$i\not\in(I\cup I')$ and therefore $l_i=l'_i$.

Assume that $i\in I$ then equation (\ref{*}) implies that $(l_i-l'_i)\g_i\in(G_k+H_{i-1})$ 
and therefore $(l_i-l'_i)=rs_i$ for some positive integer $r$, and moreover,
$$
\sum_{j\in K}(n_j-n'_j)\b_j+\sum_{j\in I}(l_j-l'_j)\g_j\in(S_k+U_{i-1}).
$$
So, there exists a $(k+i+1)$-tuple of nonnegative integers 
$\bar{d}=(a_0,\dots,a_k,c_1,\dots,c_i)$ such that 
\begin{align*}
& \bar{d} \text{ is reduced with respect to } ((S_k+U_{i-1}),s_i\g_i);\\ 
& a_j=0 \text{ if } j\not\in K,\; a_j\le (n_j-n'_j) \text{ if } j\in K;\\
& \,c_j=0 \text{ if } j\not\in I,\,\; c_j\le (l_j-l'_j) \text{ if } j\in I \text{ and } c_i\le r.
\end{align*}
Since $\prod_{j=0}^k P_j^{n_j}\prod_{j=1}^i T_j^{l_j}$ is irreducible with respect 
to $\T$, we get that $\prod_{j=0}^k P_j^{a_j}\prod_{j=1}^{i-1} T_j^{c_j}$ is  irreducible with 
respect to $\T$. Thus  $\bar{d}\in\D_i$ and 
$T_i^{c_is_i}\prod_{j=0}^k P_j^{a_j}\prod_{j=1}^{i-1} T_j^{c_j}\in \T_{i+1}$. This contradicts 
irreducibility of $\prod_{j=0}^k P_j^{n_j}\prod_{j=1}^i T_j^{l_j}$. Thus $i\not\in I$.

The same argument with $n_j$ and $l_j$ switched to $n'_j$ and $l'_j$, respectively, shows that 
$i\not\in I'$.
\end{proof}

The next statement is a straight forward consequence of   Proposition \ref{uniqueness}. 
\begin{corollary}\label{order}
Suppose that $(a_0,\dots,a_k,c_1,\dots,c_i)\neq (b_0,\dots,b_k,d_1,\dots,d_i)$ are  distinct $(k+1+i)$-tuples of nonnegative integers. If  $f=\prod_{j=0}^k P_j^{a^j}\prod_{j=1}^i T_j^{c_j}$ and $g=\prod_{j=0}^k P_j^{b_j}\prod_{j=1}^i T_j^{d_j}$ are both irreducible with respect to $\T$ then $\n(f)\neq \n(g)$.
\end{corollary}

The next corollary of Lemma \ref{group representation} is a technical statement that will be used to describe  redundant jumping polynomials.  

\begin{corollary}\label{mi-bound}
Suppose that $i,k,M\in\ZZ_{\ge 0}$ and  the $(M+1+i)$-tuple of nonnegative integers $(a_0,\dots,a_M,c_1,\dots,c_i)$ is such that $a_M>0$ and $\prod_{j=0}^M P_j^{a^j}\prod_{j=1}^i T_j^{c_j}$ is irreducible with respect to $\T$. Set $\a=\sum_{i=0}^{M}a_j\b_j+\sum_{j=1}^{i}c_j\g_j$.

If $\a\in(G_k+H_i)$ then $M\le\max(k,m_i)$.
\end{corollary}

\begin{proof}
We first observe that it is enough to prove the statement under the assumption $k\ge m_i$. Applying Lemma \ref{group representation} to $\a$ we find $n_0\in\ZZ$ and a $(k +i)$-tuple of nonnegative integers $(n_1,\dots,n_k,l_1,\dots,l_i)$ such that  $\a=\sum_{i=0}^{k}n_j\b_j+\sum_{j=1}^{i}l_j\g_j$ and $n_j<q_j$, $l_j<s_j$ for all $j>0$.

Assume for contradiction that $M>k$ and apply Lemma \ref{group representation} to $\sum_{i=0}^{k}a_j\b_j+\sum_{j=1}^{i}c_j\g_j$ to find $n'_0\in\ZZ$ and a $(k +i)$-tuple of nonnegative integers $(n'_1,\dots,n'_k,l'_1,\dots,l'_i)$ such that  $\a=\sum_{i=0}^{k}n'_j\b_j+\sum_{j=k+1}^Ma_j\b_j+\sum_{j=1}^{i}l_j\g_j$ and $n'_j<q_j$, $l'_j<s_j$ for all $j>0$. Notice that since $\prod_{j=0}^M P_j^{a^j}\prod_{j=1}^i T_j^{c_j}$ is irreducible with respect to $\T$ we also have $a_j<q_j$ for all $j>0$. Thus the $(M+i)$-tuples $(n_1,\dots,n_k,0,\dots,0,l_1,\dots,l_i)$ and $(n'_1,\dots,n'_k,a_{k+1},\dots,a_M,l'_1,\dots,l'_i)$ have to coincide. This contradicts the assumption $a_M>0$. So $M\le k$.

\end{proof}

We now recall the definition of a generating sequence of $\n$. Let $\F_{R} = \nu(R \backslash \{0\})$ be the semigroup consisting of the values of nonzero 
elements of $R$. For $\s \in\F_R$, let $I_{\s}=\{f\in R \mid \ \n(f)\ge \s\}$. 
A (possibly infinite) sequence $\{ Q_i \}_{i>0}$ of elements of $R$ is a {\it generating sequence}  of $\n$  if for every $\s\in \F_R$ the ideal $I_\s$ is generated by the set $\{\prod_{i=1}^k{Q_i}^{b_i}\mid k,\ b_i\in \ZZ_{\ge 0},\ \n(\prod_{i=1}^k{Q_i}^{b_i})\ge \s\}$. Notice that if $\s \in\F_R$ then there exist $k,b_1,\dots, b_k\in\ZZ_{\ge0}$ such that $\n(\prod_{i=1}^k Q_i^{b_i})=\s$. In particular, the set of values $\{\n(Q_i)\}_{i>0}$ generate $\F_R$ as a semigroup. A generating sequence $\{ Q_i \}_{i>0}$ is minimal if no its proper subsequence is a generating sequence of $\n$.

For $\s\in\F_R$ denote by $\A_\s$ the ideal of $R$ generated by the set
$
\{\prod_{j=0}^k P_j^{a_j}\prod_{j=1}^iT_j^{c_j}\mid 
i,k,a_j,c_j\in\ZZ_{\ge 0},\ \n(\prod_{j=0}^k P_j^{a_j}\prod_{j=1}^iT_j^{c_j})\ge\s\}$
and by $\A^+_\s$ the ideal of $R$  generated by
$\{\prod_{j=0}^k P_j^{a_j}\prod_{j=1}^iT_j^{c_j}\mid 
i,k,a_j,c_j\in\ZZ_{\ge 0},\ \n(\prod_{j=0}^k P_j^{a_j}\prod_{j=1}^iT_j^{c_j})>\s\}$.

We will show that $I_\s=\A_\s$ for all $\s\in\F_R$ to conclude that 
$\{P_j\}_{j\ge 0}\cup\{T_j\}_{j>0}$ is a generating sequence of $\n$. To this end we observe the 
following properties of $\A$-ideals 

\begin{itemize}

\item[1)] $\A_\s\subset I_\s$

\item[2)] if $\s_1<\s_2$ then $\A_{\s_2}\subset\A^+_{\s_1}\subset\A_{\s_1}$

\item[3)] $\A_{\s_1}\A_{\s_2}\subset\A_{\s_1+\s_2}$ and
$\A^+_{\s_1}\A_{\s_2}\subset\A^+_{\s_1+\s_2}$

\item[4)] if $\a=\min(\b_0,\b_1,\g_1)$ then $m_R\subset\A_\a\subset\A^+_0$

\item[5)] for any $\s\in\F_R$ there exists $\s'\in\F$ such that $\A^+_\s=\A_{\s'}$

\noindent (we take $\s'=\min\{\a\in\F_R|\ \s<\a\le\s+\b_0\}$, where $\s'$ is well defined  
since the set on the right is a nonempty finite set.)

\end{itemize}

\begin{lemma}\label{reduced}
Suppose that $k,i\ge 0$, $f\in R\setminus\{0\}$ and $f=\prod_{j=0}^kP_j^{a_j}\prod_{j=1}^iT_j^{c_j}$ for some
$a_j,c_j\in\ZZ_{\ge 0}$. For $\a=\n(\prod_{j=0}^k P_j^{a_j}\prod_{j=1}^iT_j^{c_j})$
let the $(m+1+i)$-tuple of nonnegative integers $(n_0,\dots,n_m,l_1,\dots,l_i)$ be as 
described by Proposition \ref{uniqueness}. 

Then there exist $\th\in k\setminus\{0\}$ and $f'\in \A_{\a}^+$ such that 
$f=\th\prod_{j=0}^m P_j^{n_j}\prod_{j=1}^i T_j^{l_j}+f'$.
\end{lemma}

\begin{proof}
As above it suffices to consider the case when $i\ge 0$ and $k\ge m_i$.
We set $c_0=0$ and use induction on $(i,c_i,k)$. If $i=0$ and 
$k=m_0=0$ then the statement holds with $\th=1$, $n_0=a_0$ and $f'=0$.

Assume that $i$ is fixed and $k>m_i$. Let $n_k, b\in\ZZ_{\ge 0}$ be such that 
$a_k=bq_k+n_k$ and $n_k<q_k$. Then 
$$
P_k^{a_k}=P_k^{n_k}(\l_k\prod_{j=0}^{k-1}P_j^{n_{k,j}}+P_{k+1})^b= 
(\l_k^b\prod_{j=0}^{k-1}P_j^{bn_{k,j}})P_k^{n_k}+h
$$
where $h\in\A^+_{a_k\b_k}$. Let $g=\prod_{j=0}^{k-1}P_j^{a_j+bn_{k,j}}\prod_{j=1}^iT_j^{c_j}$ 
and $h'=\prod_{j=0}^{k-1}P_j^{a_j}\prod_{j=1}^iT_j^{c_j}h$.  Notice that  $h'\in\A^+_{\a}$. 
Since $k-1\ge m_i$ the inductive hypothesis applied to $g$  gives
$$
f=\l_k^{b}gP_k^{n_k}+h'=\l_k^b(\th'\prod_{j=0}^{k-1} P_j^{n_j}\prod_{j=1}^i T_j^{l_j}+g')P_k^{n_k}+h',
$$
where $\th'\in k\setminus \{0\}$, $g'\in\A^+_{\a-n_k\b_k}$ and $\prod_{j=0}^{k-1} P_j^{n_j}\prod_{j=1}^i T_j^{l_j}$ 
is irreducible with respect to $\T$. Then by the argument from the proof of Lemma 
\ref{uniqueness} we also have $\prod_{j=0}^k P_j^{n_j}\prod_{j=1}^i T_j^{l_j}$ is irreducible with respect to $\T$, and therefore, the statement holds with $\th=\l_k^{b}\th'$ and $f'=\l_k^{b}g'P_k^{n_k}+h'$.

Assume now that $i>0$ and $k=m_i$. Let $g=\prod_{j=0}^kP_j^{a_j}\prod_{j=1}^{i-1}T_j^{c_j}$.
Then by the inductive hypothesis applied to $g$
$$
f=gT_i^{c_i}=(\th'\prod_{j=0}^kP_j^{n'_j}\prod_{j=1}^{i-1}T_j^{l'_j})T_i^{c_i}+g'T_i^{c_i},
$$ 
where $\th'\in k\setminus\{0\}$ and $g'\in\A^+_{\a-c_i\g_i}$. If  $\prod_{j=0}^k P_j^{n'_j}\prod_{j=1}^{i-1} T_j^{l'_j}T_i^{c_i}$ is irreducible with respect to $\T$ then the 
statement holds with $\th=\th'$, $f'=g'T_i^{c_i}$ and 
$(n_0,\dots,n_m,l_1,\dots,l_i)=(n'_0,\dots,n'_k,l'_1,\dots,l'_{i-1},c_i)$. If 
$\prod_{j=0}^k P_j^{n'_j}\prod_{j=1}^{i-1} T_j^{l'_j}T_i^{c_i}$ is not irreducible with respect to $\T$ then there exists
a $(k+1+i)$-tuple of nonnegative integers $\bar{d}=(a'_0,\dots,a'_k,c'_1,\dots,c'_i)$
such that $a'_j\le n'_j$, $c'_j\le l'_j$ for all $j$, $0<c'_is_i\le c_i$ and 
$(T_i^{c'_is_i}\prod_{j=0}^kP_j^{a'_j}\prod_{j=1}^{i-1}T_j^{c'_j})\in\T$. Then from definition of $T_{\bar{d}}$
it follows that 
\begin{align*}
f &=\th'(\m_{\bar{d}}\prod_{j=0}^kP_j^{n_{\bar{d},j}}\prod_{j=1}^{i-1}T_j^{l_{\bar{d},j}}+T_{\bar{d}})
(\prod_{j=0}^kP_j^{n'_j-a'_j}\prod_{j=1}^{i-1}T_j^{l'_j-c'_j})T_i^{c_i-c'_is_i}+g'T_i^{c_i}\\
&=\m_{\bar{d}}\th'(\prod_{j=0}^kP_j^{n_{\bar{d},j}+n'_j-a'_j}\prod_{j=1}^{i-1}T_j^{l_{\bar{d},j}+l'_j-c'_j })
T_i^{c_i-c'_is_i}+h+g'T_i^{c_i},
\end{align*}
where $h\in\A^+_{\a}$. Let 
$g_1=(\prod_{j=0}^kP_j^{n_{\bar{d},j}+n'_j-a'_j}\prod_{j=1}^{i-1}T_j^{l_{\bar{d},j}+l'_j-c'_j })T_i^{c_i-c'_is_i}$. 
Since $c_i-c'_is_i<c_i$ by the inductive hypothesis applied to $g_1$
$$
f=\m_{\bar{d}}\th'g_1+h+g'T_i^{c_i}=\m_{\bar{d}}\th'(\th''\prod_{j=0}^kP_j^{n_j}
\prod_{j=1}^iT_j^{l_j}+g'_1)+h+g'T_i^{c_i},
$$
where $\th''\in k\setminus \{0\}$, $g'_1\in\A^+_{\a}$ and $\prod_{j=0}^k P_j^{n_j}\prod_{j=1}^i T_j^{l_j}$  is irreducible with respect to $\T$, and therefore, the statement holds with $\th=\m_{\bar{d}}\th'\th''$ 
and $f'=\m_{\bar{d}}\th'g'_1+h+g'T_i^{c_i}$.
\end{proof}

\begin{lemma}\label{i=a}
If $\g\in\F_R$, then $I_\g=\A_\g$.
\end{lemma}

\begin{proof}
We only need to check that $I_\g\subset \A_\g$ for all $\g\in\F_R$.

Let $\g\in\F_R$ and let $f\in I_\g$. We will show that $f\in\A_\g$.
First notice that if $f\in\A_\a$ for some $\a\in\F_R$ then
$\a\le\n(f)$. Thus the set $\W=\{\a\in\F_R|\,f\in\A_\a\}$ is finite
since it is bounded from above and it is nonempty since $f\in\A_0$.
We choose $\s$ to be the maximal element of $\W$. Then there exists
a presentation
$$
f=\sum_{e=1}^N g_e\prod_{j=0}^kP_j^{a_{e,j}}\prod_{j=1}^iT_j^{c_{e,j}}+f',
$$
where $g_e\in R$,  $\sum_{j=0}^k a_{e,j}\b_j+\sum_{j=1}^i c_{e,j}\g_j=\s$ for all 
$e$ and $f'\in \A^+_\s$.

We now apply Lemma \ref{reduced} to $\prod_{j=0}^kP_j^{a_{e,j}}\prod_{j=1}^iT_j^{c_{e,j}}$
for all $1\le e\le N$. We get $\prod_{j=0}^kP_j^{a_{e,j}}\prod_{j=1}^iT_j^{c_{e,j}}=
\th_e\prod_{j=0}^m P_j^{n_j}\prod_{j=1}^i T_j^{l_j}+h_e$, where $\th_e\in k\setminus\{0\}$, $h_e\in\A^+_\s$ 
and $\sum_{j=0}^m n_j\b_j+\sum_{j=1}^i l_j\g_i=\s$. Thus
$$
f=(\sum_{e=1}^N\th_e g_e)\prod_{j=0}^m P_j^{n_j}\prod_{j=1}^i T_j^{l_j}+\sum_{e=1}^N h_e
g_e+f'= g\prod_{j=0}^m P_j^{n_j}\prod_{j=1}^i T_j^{l_j}+h,
$$
where $g=(\sum_{e=1}^N\th_e g_e)$ and  $h\in\A^+_\s$. If $g\in m_R$ then 
$ g\prod_{j=0}^m P_j^{n_j}\prod_{j=1}^i T_j^{l_j}\in\A^+_0\A_\s\subset\A^+_\s$ 
and therefore $f\in\A^+_\s$.
Let $\a\in\F_R$ be such that $\A^+_\s=\A_\a$. Then $\a>\s$ and
$f\in\A_\a$, a contradiction to the choice of $\s$. So $g$ is a
unit in $R$ and $\n(g)=0$. Since $\n(g\prod_{j=0}^m P_j^{n_j}\prod_{j=1}^i T_j^{l_j})=\s$ and $\n(h)>\s$ 
we get that $\n(f)=\s$. Thus $\s\ge\g$, and so $f\in\A_\s\subset\A_\g$.
\end{proof}

\begin{theorem}\label{generating_sequence}
$\{P_i\}_{i\ge 0}\cup\{T_i\}_{i>0}$ is a generating sequence of $\n$, $\F_R=S+U$ and $\{\b_i\}_{i\ge 0}\cup\{\g_i\}_{i>0}$ generate $\F_R$ as a semigroup.
\end{theorem}

\begin{proof}
The statement follows at once from the
definition of generating sequences and Lemma \ref{i=a}. 
\end{proof}

It is desirable to determine a minimal set of generators for $\F_R$ and to extract a minimal generating sequence for $\n$ if possible. In general $\{\b_i\}_{i\ge 0}\cup\{\g_i\}_{i>0}$ will not be a minimal set of generators for $\F_R$. One way to reduce this set is to remove dependent values: if $i>0$ we say that  $\b_i$ is dependent if $\b_i\in S_{i-1}$, we say that $\g_i$ is dependent if $\g_i\in (S_{m_i}+U_{i-1})$. For $i\ge 0$ we say that $\b_i$ (or $\g_{i+1}$) is independent if $\b_i$ (or $\g_{i+1}$) is not dependent. Then the set of all independent values  is a generating set for $\F_R$. It is not  minimal in general, see example \ref{Example2}. On the level of polynomials removing redundant jumping polynomials will lead to a subsequence  of  $\{P_i\}_{i\ge 0}\cup\{T_i\}_{i>0}$  that is still a generating sequence of $\n$. 

\begin{lemma}\label{redundant_nec}
Suppose that $i>0$ and $T_i\neq 0$ is a redundant jumping polynomial. Denote by $K$ the number of  nonzero successors of $T_i$. Then there exist $M,N\in\ZZ_{>0}$ such that
$$
T_i=\sum_{e=0}^K\th_e\prod_{j=0}^M P_j^{a_{e,j}}\prod_{j=1}^N T_j^{c_{e,j}},
$$
 where $\th_e\in k\setminus\{0\}$, $a_{e,j},c_{e,j}\in\ZZ_{\ge 0}$ for all $e,j$ and $\prod_{i=0}^M P_j^{a_{e,j}}\prod_{j=1}^N T_j^{c_{e,j}}$ is irreducible with respect to $\T$ for all $e$. Moreover, $\n(\prod_{j=0}^M P_j^{a_{0,j}}\prod_{j=1}^N T_j^{c_{0,j}})=\n(T_i)$ and $\n(\prod_{j=0}^M P_j^{a_{e-1,j}}\prod_{j=1}^N T_j^{c_{e-1,j}})<\n(\prod_{j=0}^M P_j^{a_{e,j}}\prod_{j=1}^N T_j^{c_{e,j}})$ for all $e\ge1$. 
\end{lemma}

\begin{proof}
We use induction on $K$. Since $T_i\neq0$  and $T_i$ is redundant, by definition of the only immediate successor $T_{\sum_{j=0}^i\d_j}$ of $T_i$ we have
$$
T_i=\m_{\bar{d}}\prod_{j=0}^{m_i}P_j^{n_{\bar{d},j}}\prod_{j=1}^{i-1}T_j^{l_{\bar{d},j}}+T_{\sum_{j=0}^i\d_j},
$$ 
where $\m_{\bar{d}}\in k\setminus\{0\}$, $\prod_{j=0}^{m_i}P_j^{n_{\bar{d},j}}\prod_{j=1}^{i-1}T_j^{l_{\bar{d},j}}$ is irreducible with respect to $\T$ and $\n(T_i)=\n(\prod_{j=0}^{m_i}P_j^{n_{\bar{d},j}}\prod_{j=1}^{i-1}T_j^{l_{\bar{d},j}})$ while $\n(T_i)<\n(T_{\sum_{j=0}^i\d_j})$.

If $K=0$ then $T_{\sum_{j=0}^i\d_j}=0$ and the above representation satisfies the statement of the lemma. If $K>0$ then $T_{\sum_{j=0}^i\d_j}$ is a redundant jumping polynomial with $K-1$ nonzero successors. So by the inductive hypothesis 
$$
T_i=\m_{\bar{d}}\prod_{j=0}^{m_i}P_j^{n_{\bar{d},j}}\prod_{j=1}^{i-1}T_j^{l_{\bar{d},j}}+\sum_{e=0}^{K-1}\th'_e\prod_{j=0}^{M'} P_j^{a'_{e,j}}\prod_{j=1}^{N'} T_j^{c'_{e,j}},
$$
where $\th'_e\in\ k\setminus\{0\}$, $a'_{e,j},c'_{e,j}\in\ZZ_{\ge 0}$ for all $e,j$, $\prod_{i=0}^{M'} P_j^{a'_{e,j}}\prod_{j=1}^{N'} T_j^{c'_{e,j}}$ is irreducible with respect to $\T$ for all $e$ and $\n(T_i)<\n(T_{\sum_{j=0}^i\d_j})\le\n(\prod_{j=0}^{M'} P_j^{a'_{e-1,j}}\prod_{j=1}^{N'} T_j^{c'_{e-1,j}})<\n(\prod_{j=0}^{M'} P_j^{a'_{e,j}}\prod_{j=1}^{N'} T_j^{c'_{e,j}})$ for all $e\ge1$. Thus after appropriately renaming the indices, exponents  and coefficients we obtain the required representation. 
\end{proof}
 
We now refine the notion of $\A$-ideals by using only jumping polynomials that are not redundant. Set $\R=\{j\mid T_j \text{ is redundant}\}$. For $\s\in\F_R$ denote by $\tilde{\A}_\s$ the ideal of $R$ generated by the set
$$
\{\prod_{j=0}^k P_j^{a_j}\prod_{j=1}^iT_j^{c_j}\mid 
i,k,a_j,c_j\in\ZZ_{\ge 0};\ \n(\prod_{j=0}^k P_j^{a_j}\prod_{j=1}^i{T}_j^{c_j})\ge\s;\ c_j=0 \text { for all } j\in\R\}.
$$ 

\begin{corollary}\label{a=a}
If $\s\in\F_R$ then $\tilde{A}_{\s}=\A_{\s}.$
\end{corollary}

\begin{proof}
By construction $\tilde{\A}_{\s}\subset \A_{\s}$. To show that $\A_{\s}\subset\tilde{\A}_{\s}$ we will check that any generator $\prod_{j=0}^k P_j^{a_j}\prod_{j=1}^i T_j^{c_j}$ of $\A_{\s}$ belongs to $\tilde{\A}_{\s}$. Since $\tilde{\A}_{\s_1}\tilde{\A}_{\s_2}\subset\tilde{\A}_{\s_1+\s_2}$ for all $\s_1,\s_2,\in\F_R$ it is suffices to check that $P_i\in \tilde{\A}_{\b_i}$ and $T_i\in \tilde{\A}_{\g_i}$ for all $i$. By construction we have $P_i\in\tilde{\A}_{\b_i}$ for all $i\in \ZZ_{\ge 0}$ and $T_i\in\tilde{\A}_{\g_i}$ for all $i\in\ZZ_{>0}\setminus\R$. 

Assume that $i\in\R$ and use Lemma \ref{redundant_nec} to write
$$T_i=\sum_{e=0}^K\th_e\prod_{j=0}^M P_j^{a_{e,j}}\prod_{j=1}^N T_j^{c_{e,j}}.$$ 
For all $e\ge 0$, since $\prod_{j=0}^M P_j^{a_{e,j}}\prod_{j=1}^N T_j^{c_{e,j}}$ is irreducible with respect to $\T$ we have $c_{e,j}=0$ for all $j\in\R$. Also for all $e\ge 0$, we have $\n(\prod_{j=0}^M P_j^{a_{e,j}}\prod_{j=1}^N T_j^{c_{e,j}})\ge\g_i$. Thus $T_i\in\tilde{\A}_{\g_i}$. 
\end{proof}

The next statement follows immediately from Lemma \ref{i=a}, Corollary \ref{a=a} and  definition of generating sequences.

\begin{proposition}\label{reduced_generating_sequence}
$\{P_i\}_{i\in\ZZ_{\ge 0}}\cup\{T_i\}_{i\in\ZZ_{>0}\setminus\R}$ is a generating sequence of $\n$.
\end{proposition}

We now provide a sufficient condition for a jumping polynomial to be redundant. It will allow us to recognize redundant jumping polynomials in Example \ref{Example2}.  
\begin{lemma}\label{redundant_suf}
Suppose that $i,K,M>0$ and $T_i=\sum_{e=1}^K\th_e\prod_{j=0}^M P_j^{a_{e,j}}\prod_{j=1}^{i-1}T_j^{c_{e,j}}\neq 0$, with $\th_e\in k$, $a_{e,j},c_{e,j}\in\ZZ_{\ge 0}$ for all $e,j$ and $\prod_{j=0}^M P_j^{a_{e,j}}\prod_{j=1}^{i-1}T_j^{c_{e,j}}$ irreducible with respect to $\T$ for all $e$.  Then $T_i$ is redundant.
\end{lemma}

\begin{proof} 
After possibly collecting like terms we may assume that  in the presentation
$$
T_i=\sum_{e=1}^K\th_e\prod_{i=0}^M P_j^{a_{e,j}}\prod_{j=1}^{i-1}T_j^{c_{e,j}}
$$
 $(a_{e_1,0},\dots,a_{e_1,M},c_{e_1,1},\dots,c_{e_1,i-1})\neq(a_{e_2,0},\dots,a_{e_2,M},c_{e_2,1},\dots,c_{e_2,i-1})$ for all $e_1\neq e_2$and $\th_e\neq 0$ for all $e$. Then by Corollary \ref{order}, after possibly rearranging terms, we may further assume that $\n(\prod_{i=0}^M P_j^{a_{e-1,j}}\prod_{j=1}^{i-1}T_j^{c_{e-1,j}})<\n(\prod_{i=0}^M P_j^{a_{e,j}}\prod_{j=1}^{i-1}T_j^{c_{e,j}})$ for all $e>1$ and $\n(\prod_{i=0}^M P_j^{a_{1,j}}\prod_{j=1}^{i-1}T_j^{c_{1,j}})=\g_i$. In particular, $\g_i\in(S_M+U_{i-1})\subset(G+H_{i-1})$.
 
By construction of jumping polynomials we have  $s_i=1$ and $\g_i\in(G_{m_i}+H_{i-1})$. Denote by $M_1=\max\{j\mid a_{1,j}>0\}$ and apply Corollary \ref{mi-bound} to $\g_i=\sum_{j=0}^{M_1}a_{1,j}\b_j+\sum_{j=1}^{i-1}c_{1,j}\g_j$ to get $M_1\le m_i$. Thus $\g_i\in(S_{m_i}+U_{i-1})$. 
Also, the only immediate successor of  $T_i$ is $T_{\sum_{j=0}^i\d_i}=T_i-\th_1\prod_{i=0}^{m_i} P_j^{a_{1,j}}\prod_{j=1}^{i-1}T_j^{c_{1,j}}$.
 
 We now use induction on $K$. If $K=1$ then  $T_{\sum_{j=0}^i\d_i}=0$ and, therefore, $T_i$ is redundant. If $K>1$ then  $T_{\sum_{j=0}^i\d_i}= \sum_{e=2}^K\th_e\prod_{i=0}^M P_j^{a_{e,j}}\prod_{j=1}^{i-1}T_j^{c_{e,j}}$. By the inductive hypothesis $T_{\sum_{j=0}^i\d_i}$ is redundant  and, therefore, $T_i$ is redundant. 
\end{proof}


\section{Defining sequences}\label{numerical data}

In this section we describe numerical data that uniquely determines a 
valuation on $k(x,y,z)$ centered at $k[x,y,z]_{(x,y,z)}$.

Given sequences of positive rational numbers $\{\b'_i\}_{i\ge 0}$ and 
$\{\bar{\g}_i\}_{i>0}$ we use notation as in section \ref{preliminaries}. Let $\bar{m}_0=0$, $G'_0=\b'_0\ZZ$, $\bar{H}_0=\{0\}$ and for all $i> 0$ set 
\begin{align*}
S'_i=\sum_{j=0}^i\b'_j\ZZ_{\ge 0},\quad G'_i=\sum_{j=0}^i\b'_j\ZZ\quad&\text{ and }\quad S'=\bigcup_{j=0}^{\infty}S'_j,\quad G'=\bigcup_{j=0}^{\infty}G'_j\\
\bar{U}_i=\sum_{j=1}^i \bar{\g}_j\ZZ_{\ge 0},\quad \bar{H}_i=\sum_{j=1}^i \bar{\g}_j\ZZ\quad &\text{ and }\quad \bar{U}=\bigcup_{j=1}^{\infty}\bar{U}_j,\quad \bar{H}=\bigcup_{j=1}^{\infty}\bar{H}_j\\
q'_i=\min\{q\in\ZZ_{>0}\mid q\b'_i\in G'_{i-1}\},\quad\quad &\; \bar{s}_i=\min\{s\in\ZZ_{>0}\mid s\bar{\g}_i\in(\bar{H}_{i-1}+G')\}\\
 \text {and }\quad \bar{m}_i =\max(\bar{m}_{i-1},  \min &\{j\in\ZZ_{\ge 0}\mid  \bar{s}_i\bar{\g}_i\in(\bar{H}_{i-1}+G'_j)\}).
\end{align*}  

Applying Lemma \ref{group representation} to $\a=q'_i\b'_i$, an element of $G'_{i-1}$ we find $n'_{i,0}\in\ZZ$ and nonnegative integers $n'_{i,1},\dots,n'_{i,i-1}$ such that $n'_{i,j}<q'_j$ for all $j>0$ and $\a=\sum_{j=0}^{i-1}n'_{i,j}\b'_j$.  Applying Lemma \ref{group representation} to $\a=\bar{s}_i\bar{\g}_i$, an element of $G'_{\bar{m}_i}+\bar{H}_{i-1}$, we find $a_i\in\ZZ$ and nonnegative integers $\bar{n}_{i,1},\dots,\bar{n}_{i,\bar{m}_i},\bar{l}_{i,1},\dots,\bar{l}_{i,i-1}$ such that $\bar{n}_{i,j}<q'_j$ and $\bar{l}_{i,j}<\bar{s}_j$ for all $j>0$ and
$\a=a_i\b'_0+\sum_{j=1}^{\bar{m}_i}\bar{n}_{i,j}\b'_j+\sum_{j=1}^{i-1}\bar{l}_{i,j}\bar{\g}_j$. We set $\bar{n}_{i,0}=\max(0,a_i)$ and $\bar{r}_{i,0}=\max(0,-a_i)$.

The conditions we require sequences $\{\b'_i\}_{i\ge 0}$ and 
$\{\bar{\g}_i\}_{i>0}$ to satisfy are

$$
\b'_{i+1}>q'_i\b'_i,
$$
$$
\bar{\g}_{i+1}>\bar{r}_{i,0}\b'_0+\bar{s}_i\bar{\g}_i.
$$
Then by Corollary \ref{positivity1} we have $n'_{i,0}>0$  for all $i>0$. 

Given sequences of residues $\{\l'_i\}_{i>0}\subset k\setminus\{0\}$ and
$\{\bar{\m}_i\}_{i>0}\subset k\setminus\{0\}$ we inductively define the sequences of polynomials
$\{P'_i\}_{i\ge 0}$ and $\{Q_i\}_{i> 0}$ by setting
\begin{align*} 
&P'_0 =x,\quad\; P'_1=y,\quad\; Q_1=z,\\
&P'_{i+1} =(P'_i)^{q'_i}-\l'_i\prod_{j=0}^{i-1}(P'_j)^{n'_{i,j}},\\
&Q_{i+1} =x^{\bar{r}_{i,0}}Q_i^{\bar{s}_i}-\bar{\m}_i\prod_{j=0}^{\bar{m}_i}
(P'_j)^{\bar{n}_{i,j}}\prod_{j=1}^{i-1}Q_j^{\bar{l}_{i,j}}.
\end{align*}

Our main statement to be proved in section \ref{properties2} is

\begin{theorem}\label{main}
In the above notation assuming that infinitely many $q'_i$ and $\bar{s}_i$ are greater than 1 there exists a unique valuation $\nu: k(x,y,z)\ra \QQ$ such that $\nu$ dominates $k[x,y,z]_{(x,y,z)}$ and $\nu(P'_i)=\b'_i$ and $\nu(Q_i)=\bar{\g}_i$ for all $i$.
\end{theorem}

If $\n$ is a valuation on $k[x,y,z]_{(x,y,z)}$ as provided by Theorem \ref{main} we claim that the sequence of jumping polynomials $\{P_i\}_{i\ge 0}\cup\{T_i\}_{i>0}$ for $\n$ as defined in section \ref{construction} satisfies $P_i=P'_i$ for all $i\ge 0$. Indeed, we have $P_0=x=P'_0$ and $P_1=y=P'_1$. Fix $i\in\ZZ_{>0}$ and assume that $P_j=P'_j$ for all $j\le i$.  Then $\b_j=\b'_j$, $G_j=G'_j$ and $q_j=q'_j$ for all $j\le i$. Due to uniqueness of $n_{i,0},\dots,n_{i,i-1}$ such that $q'_i\b'_i=q_i\b_i=\sum_{j=0}^{i-1}n_{i,j}\b_{i,j}$ and $n_{i,j}<q_j$ for all $j>0$ we also have $n_{i,j}=n'_{i,j}$ for all $j<i$.  Then $P_i^{q_i}/\prod_{j=0}^{i-1}P_j^{n_{i,j}}=(P'_i)^{q'_i}/\prod_{j=0}^{i-1}(P'_j)^{n'_{i,j}}$ and therefore, $\l_i=\l'_i$. Thus 
$$
P_{i+1}=P_i^{q_i}-\l_i\prod_{j=1}^{i-1}P_j^{n_{i,j}}=(P'_i)^{q'_i}-\l'_i\prod_{j=1}^{i-1}(P'_j)^{n'_{i,j}}=P'_{i+1}.
$$
We will now drop apostrophes in the notation for defining polynomials.


\section{Properties of defining polynomials}\label{properties2}

The goal of this section is to prove Theorem \ref{main}. We use the simplified notation of section \ref{numerical data}. In particular, if $i\ge 0$ then $P_i$ denotes a defining polynomial. Also, we use lexicographical order to compare $k$-tuples of integers: we say $(a_1,\dots,a_k)<(b_1,\dots,b_k)$ if and only if there exists $l<k$ such that $a_i=b_i$ for all $i\le l$ and $a_{l+1}<b_{l+1}$. 
 
\begin{proposition}\label{P-degrees} 
Suppose that $A=k[x]$ and $i\ge 1$ then $P_i$ is a monic polynomial in $A[y]$ and $\deg_y P_i=\prod_{j=1}^{i-1}q_j$. Moreover, if $a_1,a_2,\dots, a_{i-1}$ are nonnegative integers such that $a_j<q_j$ for all $j<i$ then $\deg_y(\prod_{j=1}^{i-1}P_j^{a_j})<\deg_y P_i$.
\end{proposition}

\begin{proof}
Same argument as in the proof of Proposition \ref{P-degrees1} proves this statement. 
\end{proof}

\begin{corollary}\label{new lemma}
Suppose that the $m$-tuples of nonegative integers $(a_1,\dots,a_m)$ and $(b_1,\dots,b_m)$ are such that $a_j,b_j<q_j$ for all $j<m$. If $(a_m,\dots,a_1)<(b_m,\dots,b_1)$ then $\deg_y (\prod_{j=1}^m P_j^{a_j})<\deg_y (\prod_{j=1}^m P_j^{b_j})$. 

In particular, if
$f=\sum_{a_m=0}^q\left(\sum_{0\le a_j<q_j}f_{a_1,\dots,a_m}(x)\prod_{j=1}^{m-1} P_j^{a_j}\right)P_m^{a_m},$
 where $f_{a_1,\dots,a_m}\in k[x]$ for all $(a_1,\dots,a_m)$,  then $\deg_y (\prod_{j=1}^m P_j^{a_j})\le \deg_y f$.  
\end{corollary}

\begin{proof}
Let $(a_m,\dots,a_1)<(b_m,\dots,b_1)$ then there exists $1\le l\le m$ such that $a_l<b_l$ and $a_j=b_j$ for all $j>l$. Then 
$$\deg_y(\prod_{j=1}^m P_j^{a_j}) < \deg_y(P_l^{a_l+1}\prod_{j=l+1}^m P_j^{a_j})\le \deg_y(\prod_{j=l}^m P_j^{b_j})\le\deg_y(\prod_{j=1}^m P_j^{b_j}).
$$

Consider now 
$f=\sum_{a_m=0}^q\left(\sum_{0\le a_j<q_j}f_{a_1,\dots,a_m}(x)\prod_{j=1}^{m-1} P_j^{a_j}\right)P_m^{a_m}$. Set 
$$(b_m,\dots,b_1)=\max\{(a_m,\dots,a_1)\mid f_{a_1,\dots,a_m}\neq 0\}$$ 
Then $\deg_y(\prod_{j=1}^m P_j^{a_j})<\deg_y(\prod_{j=1}^m P_j^{b_j})$ for all $(a_1,\dots,a_m)\neq (b_1,\dots, b_m)$. Thus $\deg_y f=\deg_y (\prod_{j=1}^m P_j^{b_j})\ge \deg_y(\prod_{j=1}^m P_j^{a_j})$. Notice also that $\lead_y f=f_{b_1,\dots,b_m}$.
\end{proof}

\begin{proposition}\label{Q-degrees-intermediate}
Suppose that $A=k[x,y]$ and $i\ge 1$ then $Q_i$ is a polynomial in $A[z]$ with
 $\deg_z Q_i= \prod_{j=1}^{i-1}\bar{s}_j$ and $\lead_z Q_i=x^{d_i}$, 
 where $d_i=\sum_{k=1}^{i-1}(\bar{r}_{k,0}\prod_{j=k+1}^{i-1}\bar{s}_j)$.  
 
Moreover, if $a_1,a_2,\dots, a_{i-1}$ are nonnegative integers such that $a_j<\bar{s}_j$ for all $j<i$ then $\deg_z(\prod_{j=1}^{i-1}Q_j^{a_j})<\deg_z Q_i$.
\end{proposition}
 
\begin{proof} 
We use induction on $i$. For $i=1$ we have $Q_1=z$ is monic of degree 1. Let $i>1$ then 
$Q_i=x^{\bar{r}_{i-1,0}}Q_{i-1}^{\bar{s}_{i-1}}-\bar{\m}_{i-1}\prod_{j=0}^{\bar{m}_{i-1}}P_j^{\bar{n}_{i-1,j}}\prod_{j=1}^{i-2}Q_j^{\bar{l}_{i-1,j}}$. 
By the inductive hypothesis we have  
\[
\begin{split}
\deg_z Q_{i-1}^{\bar{s}_{i-1}} = \prod_{j=1}^{i-1} & \bar{s}_j,\\
\deg_z (\prod_{j=1}^{i-2}Q_j^{\bar{l}_{i-1,j}})\, & =\bar{l}_{i-1,1}+\bar{l}_{i-1,2}\bar{s}_1+\dots+\bar{l}_{i-1,i-2}\bar{s}_1\cdots \bar{s}_{i-3}\\
& <\bar{s}_1+(\bar{s}_2-1)\bar{s}_1+\dots+(\bar{s}_{i-2}-1)\bar{s}_1\cdots \bar{s}_{i-3}=\bar{s}_1\cdots \bar{s}_{i-2}.
\end{split}
\]
Thus $\deg_z Q_i=\prod_{j=1}^{i-1} \bar{s}_j$ and $\lead_z P_i=x^{\bar{r}_{i-1,0}}(\lead_z Q_{i-1})^{\bar{s}_{i-1}}=x^{d_i}$, where $d_i=\bar{r}_{i-1,0}+\bar{s}_{i-1}\sum_{k=1}^{i-2}(\bar{r}_{k,0}\prod_{j=k+1}^{i-2}\bar{s}_j)=\sum_{k=1}^{i-1}(\bar{r}_{k,0}\prod_{j=k+1}^{i-1}\bar{s}_j)$.
\end{proof} 
 
 The next statement is a  straight forward consequence of Proposition \ref{Q-degrees-intermediate}. 
 
\begin{corollary}\label{Q-degrees}
Suppose that $A=k[x,x^{-1},y]$  and $i\ge 1$ then $x^{-d_i}Q_i$ is a monic polynomial in $A[z]$ and $\deg_z (x^{-d_i}Q_i)=\bar{s}_1\cdots \bar{s}_{i-1}$.
\end{corollary}

Using  the argument of the proof of Corollary \ref{new lemma} we also get the   following statement:
\begin{corollary}\label{new q-lemma}
Suppose that the $n$-tuples of nonegative integers $(a_1,\dots,a_n)$ and $(b_1,\dots,b_n)$ are such that $a_j,b_j<\bar{s}_j$ for all $j<n$. If $(a_n,\dots,a_1)<(b_n,\dots,b_1)$ then $\deg_z (\prod_{j=1}^n Q_j^{a_j})<\deg_z (\prod_{j=1}^n Q_j^{b_j})$. 

In particular, if
$f=\sum_{a_n=0}^s\left(\sum_{0\le a_j<\bar{s}_j}f_{a_1,\dots,a_n}(x,y)\prod_{j=1}^{n-1} Q_j^{a_j}\right)Q_n^{a_n},$
 where $f_{a_1,\dots,a_n}\in k[x,y]$ for all $(a_1,\dots,a_n)$,  then $\deg_z (\prod_{j=1}^n Q_j^{a_j})\le\deg_z f$.
\end{corollary}

\begin{remark} We expect the sequence of polynomial $\{P_i\}_{i>0}$ to satisfy MacLane's axioms for key polynomials corresponding to the field extension $k(x)\hookrightarrow k(x,y)$ and $\{x^{-d_i}Q_i\}_{i>0}$ to satisfy the axioms for key polynomials corresponding to the field extension $k(x,y)\hookrightarrow k(x,y,z)$. 
\end{remark}

\begin{lemma}\label{lemma1}
Suppose that $f\in k[x,y]$ and $M$ is such that $\deg_y f<\deg_y P_{M+1}$ then there exists a unique representation 
\begin{equation}\label{P-expansion}
f=\sum_{0\le a_j<q_j}f_{a_1,\dots,a_M}(x)\prod_{j=1}^M P_j^{a_j},
\end{equation}
where $f_{a_1,\dots,a_M}\in k[x]$ for all $(a_1,\dots,a_M)$.
\end{lemma}

\begin{proof}
We use induction on $M$. If $M=0$ then $\deg_y f<1$. Thus $f\in k[x]$ and $f=f(x)$ is the required representation. 

Assume $M>0$ and set $A=k[x]$. If $f\neq 0$ since $P_M$ is a monic polynomial in $A[y]$ by Euclidean division in $A[y]$ we find $q\in\ZZ_{\ge 0}$ such that 
\begin{equation}\label{powerexpansion}
f=\sum_{i=0}^q g_iP_M^i,
\end{equation}
where $g_i\in A[y]$, $g_q\neq 0$ and $\deg_y g_i<\deg_y P_M$ for all $i$. Since $\deg_y f=\deg_y(g_qP_M^q)$ and $\deg_y f<\deg_y P_{M+1}$ applying Proposition \ref{P-degrees} we get $q\prod_{j=1}^{M-1}q_j<\prod_{j=1}^M q_j$ and therefore $q<q_M$. Thus there exists a representation 
\begin{equation}\label{completepowers}
f=\sum_{i=0}^{q}g_iP_M^i+\sum_{i=q+1}^{q_M-1}0\cdot P_M^i=\sum_{0\le a_M<q_M}h_{a_M}P_M^{a_M}
\end{equation}
such that $h_{a_M}\in A[y]$ and $\deg_y h_{a_M}<\deg P_M$ for all $a_M$.  Also $h_{a_M}$ are uniquely determined for all $a_M$ due to uniqueness of representation (\ref{powerexpansion}).

Applying the inductive hypothesis to $h_{a_M}$ for all $a_M$ we get 
$$
f=\sum_{0\le a_M<q_M}\left(\sum_{0\le a_j<q_j}h_{(a_1,\dots,a_{M-1}),a_M}(x)\prod_{j=1}^{M-1}P_j^{a_j}\right)P_M^{a_M}=\sum_{0\le a_j<q_j}f_{a_1,\dots,a_M}(x)\prod_{j=1}^MP_j^{a_j}.
$$
By the inductive hypothesis  $f_{a_1,\dots,a_M}(x)$  are uniquely determined for all $(a_1,\dots,a_M)$.
\end{proof}

In general, if $f\in k[x,y]\setminus\{0\}$ and $M$ is any positive integer following the proof of Lemma \ref{lemma1} we obtain a unique representation 
$$
f=\sum_{a_M=0}^{q'}\left(\sum_{0\le a_j<q_j}f_{a_1,\dots,a_M}(x)\prod_{j=1}^{M-1} P_j^{a_j}\right)P_M^{a_M},
$$
where $q'$ is a nonnegative integer not necessarily  less than $q_M$, $f_{a_1,\dots,a_M}\in k[x]$ for all $(a_1,\dots,a_M)$ and $f_{a_1,\dots,a_{M-1},q'}\neq 0$ for some  $(a_1,\dots,a_{M-1})$. For compatibility with representation (\ref{P-expansion}) we set $q=\max(q',q_M-1)$ and $f_{a_1,\dots,a_M}=0$ for all $(a_1,\dots,a_M)$ such that $q'<a_M\le q$. The representation 
$$
f=\sum_{a_M=0}^{q}\left(\sum_{0\le a_j<q_j}f_{a_1,\dots,a_M}(x)\prod_{j=1}^{M-1} P_j^{a_j}\right)P_M^{a_M},
$$
is called a $P_M$-expansion of $f$. Notice that if $M$ is such that $\deg_y P_{M+1}>\deg_y f$ then the $P_M$-expansion of $f$ coincides with representation (\ref{P-expansion}).

\begin{lemma}\label{lemma2-intermediate}
Suppose that $f\in k[x, x^{-1},y,z]$ and $N$ is such that $\deg_z f<\deg_z Q_{N+1}$ then there exists a  unique representation 
\begin{equation}\label{fractionalexpansion}
f=\sum_{0\le c_j<\bar{s}_j}f'_{c_1,\dots,c_N}(x,x^{-1},y)\prod_{j=1}^N (x^{-d_j}Q_j)^{c_j}
\end{equation}
where $f_{c_1,\dots,c_N}\in k[x,x^{-1},y]$ for all $(c_1,\dots,c_N)$.
\end{lemma}

\begin{proof} Notice that if $A=k[x,x^{-1},y]$ and $i>0$ then by Corollary \ref{Q-degrees} $x^{-d_i}Q_i$ is a monic polynomial in $A[z]$. Now the argument of the proof of Lemma \ref{lemma1} applies. 
\end{proof}

\begin{corollary}\label{lemma2}
Suppose that $f\in k[x,y,z]$ and $N$ is such that $\deg_z f<\deg_z Q_{N+1}$ then there exists a unique $K\in\ZZ_{\ge 0}$ and  representation 
\begin{equation}\label{Q-expansion}
 x^Kf=\sum_{0\le c_j<\bar{s}_j}f_{c_1,\dots,c_N}(x,y)\prod_{j=1}^N Q_j^{c_j},
\end{equation}
where $f_{c_1,\dots,c_N}\in k[x,y]$ for all $(c_1,\dots,c_N)$ and if $f_{c_1,\dots,c_N}\in xk[x,y]$ for all $(c_1,\dots,c_N)$ then $K=0$.
\end{corollary}

\begin{proof}
We get the required representation by multiplying both sides of representation (\ref{fractionalexpansion}) by an appropriate power of $x$. More precisely, 
$$
K=\max(\max_{(c_1,\dots,c_N)}\{\sum_{j=1}^N d_jc_j-\ord_xf'_{c_1,\dots,c_N}\},0).
$$
\end{proof}

In general, if $f\in k[x,y,z]\setminus\{0\}$ and $N$ is any positive integer we  can find a unique $K\in\ZZ_{\ge 0}$ and representation 
$$
x^Kf=\sum_{c_N=0}^{s'}\left(\sum_{0\le c_j<\bar{s}_j}f_{c_1,\dots,c_N}(x,y)\prod_{j=1}^{N-1} Q_j^{c_j}\right)Q_N^{c_N},
$$
such that  $s'$ is a nonnegative integer not necessarily  less than $\bar{s}_N$,
$f_{c_1,\dots,c_N}\in k[x,y]$ for all $(c_1,\dots,c_N)$ and $f_{c_1,\dots,c_{N-1},s'}\neq 0$ for some  $(c_1,\dots,c_{N-1})$, 
and if   $f_{c_1,\dots,c_N}\in xk[x,y]$ for all $(c_1,\dots,c_N)$ then $K=0$. For compatibility with representation (\ref{Q-expansion}) we set
$s=\max(s',\bar{s}_N-1)$ and $f_{c_1,\dots,c_N}=0$ for all $(c_1,\dots,c_N)$ such that $s'<c_N\le s$.  The representation 
$$
x^Kf=\sum_{c_N=0}^{s}\left(\sum_{0\le c_j<\bar{s}_j}f_{c_1,\dots,c_N}(x,y)\prod_{j=1}^{N-1} Q_j^{c_j}\right)Q_N^{c_N}
$$
is called a $Q_N$-expansion of $f$. Notice that if $N$ is such that $\deg_z Q_{N+1}>\deg_z f$ then the $Q_N$-expansion of $f$ coincides with representation (\ref{Q-expansion}).

\begin{theorem}\label{theorem1} Suppose that in notation of section \ref{numerical data} infinitely many $q_i$ and  $\bar{s}_i$ are greater than 1.
Then for $f\in k[x,y,z]\setminus\{0\}$ there exist unique $K,M,N\in\ZZ_{\ge 0}$ and  representation
\begin{equation}\label{NUrepresentation}
x^Kf=\sum_{\substack{0\le a_i<q_i\\
0\le c_j<\bar{s}_j}}f_{AC}(x)\prod_{i=1}^MP_i^{a_i}\prod_{j=1}^NQ_j^{c_j}, 
\end{equation}
where  $AC=(a_1,\dots,a_M,c_1,\dots,c_N)$, that satisfies the following conditions
\begin{itemize}
 \item[] $f_{AC}\in k[x]$ for all $AC$, 
 \item[] $f_{AC}\neq 0$ for some $AC$ with $c_N> 0$, 
\item[] $f_{AC}\neq 0$ for some $AC$ with $a_M> 0$,
 \item[]if $f_{AC}\in xk[x]$ for all $AC$ then $K=0$.
  \end{itemize}
\end{theorem}

\begin{proof} To construct representation (\ref{NUrepresentation}) we will first apply Corollary \ref{lemma2} to $f$ and then apply Lemma \ref{lemma1} to every term $f_{c_1,\dots,c_N}$ of representation (\ref{Q-expansion}).

First notice that since there are infinitely many $\bar{s}_i$ greater than 1, for any fixed degree $d$ there is $i$ such that $\deg_z Q_i>d$. If $f\in k[x,y]$ set $N=0$. If $f\not\in k[x,y]$ let $N$ be such that $\deg_z Q_N\le \deg_z f$ and $\deg_z f<\deg_z Q_{N+1}$. Then consider representation (\ref{Q-expansion}) of $f$
$$
 x^Kf=\sum_{0\le c_j<\bar{s}_j}f_{c_1,\dots,c_N}(x,y)\prod_{j=1}^N Q_j^{c_j}.
$$
If $f_{c_1,\dots,c_N}=0$ for all $(c_1,\dots,c_N)$ such that $c_N>0$, then $\deg_z f$ is bounded by $\deg_z\left(\prod_{j=1}^{N-1}Q_j^{\bar{s}_j-1}\right)$.  Then $\deg_z f<\deg_z Q_N$, which contradicts the choice of $N$. Thus, there exists $(c_1,\dots,c_N)$ such that $c_N>0$ and $f_{c_1,\dots,c_N}\neq 0$.

Now  since there are infinitely many $q_i$ greater than 1, for any fixed degree $d$ there exists $i$ such that $\deg_y P_i>d$. Set
$$
d=\max_{(c_1,\dots,c_N)}\deg_y f_{c_1,\dots,c_N}
$$ 
and let $M$ be such that $\deg_y P_M\le d$ and $d<\deg_y P_{M+1}$. Then consider representation (\ref{P-expansion}) of $f_{c_1,\dots,c_N}$ for all $C=(c_1,\dots,c_N)$
$$
f_C=f_{c_1,\dots,c_N}=\sum_{0\le a_i<q_i}f_{a_1,\dots,a_M,c_1,\dots,c_N}(x)\prod_{i=1}^M P_i^{a_i}=\sum_{0\le a_i<q_i}f_{AC}(x)\prod_{i=1}^M P_i^{a_i}.
$$
Observe that if $C=(c_1,\dots,c_N)$ is such that $c_N>0$ and $f_C\neq 0$ then there exists  $A=(a_1,\dots,a_M)$ such that $f_{AC}\neq 0$. 

Let now $C_d=(c_1,\dots,c_N)$ be such that $\deg_y f_{C_d}=d$. If $f_{AC_d}=0$ for all $A=(a_1,\dots,a_M)$ such that $a_M>0$, then $\deg_y f_{C_d}$ is bounded by $\deg_y\left(\prod_{i=1}^{M-1}P_i^{q_i-1}\right)$. Then $d<\deg_y P_M$, which contradicts the choice of $M$. Thus, there exists $A=(a_1,\dots,a_M)$ such that $f_{AC_d}\neq 0$ and $a_M>0$. 

Finally, if $f_{AC}\in xk[x]$ for all $AC$ then $f_{c_1,\dots,c_N}\in xk[x,y]$ for all $(c_1,\dots,c_N)$ and it follows from properties of representation (\ref{Q-expansion}) that $K=0$. Thus
$$
x^Kf=\sum_{\substack{0\le a_i<q_i\\
0\le c_j<s_j}}f_{AC}(x)\prod_{i=1}^MP_i^{a_i}\prod_{j=1}^NQ_j^{c_j}
$$
is  a representation of $f$ that satisfies all the conditions.

Assume that 
$$
x^{K'}f=\sum_{\substack{0\le a_i<q_i\\
0\le c_j<\bar{s}_j}}f'_{AC}(x)\prod_{i=1}^{M'}P_i^{a_i}\prod_{j=1}^{N'}Q_j^{c_j}
$$
is another representation of $f$ satisfying the conditions of the theorem. 

We notice that $\deg_z f\le\deg_z\left(\prod_{j=1}^{N'}Q_j^{\bar{s}_j-1}\right)<\deg_z Q_{N'+1}$. On the other hand since there exists $AC$ such that $c_{N'}>0$ and $f'_{AC}\neq 0$ we have $\deg_z f \ge\deg_z Q_{N'}$. Thus $N'=N$ .  For all $C=(c_1,\dots, c_{N})$ set 
$$f'_C=f'_{c_1,\dots,c_{N}}=\sum_{0\le a_i<q_i}f'_{AC}(x)\prod_{i=1}^{M'} P_i^{a_i}.$$

We will now show that if $K'>0$ there exists $C_0=(c_1,\dots,c_{N})$ such that $f'_{C_0}\not\in xk[x,y]$. Assume that $K'>0$, then there exists $AC$ such that $f'_{AC}\not\in xk[x]$. Let $A_0=(b_1,\dots,b_{M'})$ be such that
$$
(b_{M'},\dots,b_1)=\max_{f'_{AC}\not\in xk[x]}(a_{M'},\dots,a_1),
$$
and let $C_0=(c_1,\dots,c_{N})$ be such that $f'_{A_0C_0}\not\in xk[x]$.  Assume for contradiction that  $f'_{C_0}\in xk[x,y]$. Since for all $A=(a_1,\dots,a_{M'})$ such that $(a_{M'},\dots,a_1)>(b_{M'},\dots,b_1)$ we have $f'_{AC_0}\in xk[x]$, the following polynomial
$$
h=\sum_{\substack{A=(a_1,\dots,a_{M'})\\
(a_{M'},\dots,a_1)>(b_{M'},\dots,b_1)}}f'_{AC_0}(x)\prod_{i=1}^{M'} P_i^{a_i}
$$
is an element of $xk[x,y]$. Thus $(f'_{C_0}-h)\in xk[x,y]$, and therefore, $\lead_y(f'_{C_0}-h)\in xk[x]$. This contradicts the choice of $A_0C_0$ since by the proof of Corollary \ref{new lemma} we have $\lead_y(f'_{C_0}-h)=f'_{A_0C_0}$. 

Thus
$$
x^{K'}f=\sum_{0\le c_j<\bar{s}_j}f'_{c_1,\dots,c_{N}}(x,y)\prod_{j=1}^{N} Q_j^{c_j}
$$
is another representation of $f$ in the form of (\ref{Q-expansion}). Due to uniqueness of such a representation  $K'=K$ and $f'_{c_1,\dots,c_N}=f_{c_1,\dots,c_N}$ for all $(c_1,\dots,c_N)$.

We notice that $\deg_y f_{c_1,\dots,c_N}\le\deg_y\left(\prod_{i=1}^{M'}P_i^{q_i-1}\right)<\deg_y P_{M'+1}$ for all $(c_1,\dots,c_N)$. On the other hand  if 
$AC=(a_1,\dots,a_{M'},c_1,\dots,c_N)$ is such that $a_{M'}>0$ and $f'_{AC}\neq 0$ we have $\deg_y f_{c_1,\dots,c_N}\ge \deg_y P_{M'}$. Thus $M'=M$. For all $(c_1,\dots,c_N)$ we have 
$$
f_{c_1,\dots,c_{N}}=\sum_{0\le a_i<q_i}f'_{AC}(x)\prod_{i=1}^{M} P_i^{a_i}
$$
is another representation of $f_{c_1,\dots,c_N}$ in the form of (\ref{P-expansion}). Due to uniqueness of such a representation we get $f'_{AC}=f_{AC}$ for all $AC$.
\end{proof}

In general, if $f\in k[x,y,z]\setminus\{0\}$ and $M,\,N$ are some positive integers we  can find a unique $K\in\ZZ_{\ge 0}$ and representation 
$$
x^Kf=\sum_{c_N=0}^{s'}\sum_{a_M=0}^{q'}\left(\sum_{\substack{0\le a_i<q_i\\
0\le c_j<\bar{s}_j}}f_{AC}(x)\prod_{i=1}^{M-1}P_i^{a_i}\prod_{j=1}^{N-1}Q_j^{c_j}\right)P_M^{a_M}Q_N^{c_N},
$$
where  $AC=(a_1,\dots,a_M,c_1,\dots,c_N)$, $s',\,q'$ are nonnegative integers not necessarily  satisfying $s'<\bar{s}_N$ and $q'<q_N$, and the following conditions hold
\begin{itemize}
\item[] $f_{AC}\in k[x]$ for all $AC$, 
\item[] $f_{AC}\neq 0$ for some $AC$ with $c_N=s'$, 
\item[] $f_{AC}\neq 0$ for some $AC$ with $a_M=q'$,
\item[]if $f_{AC}\in xk[x]$ for all $AC$ then $K=0$.
\end{itemize}

For compatibility with representation (\ref{NUrepresentation}) we set
$s=\max(s',\bar{s}_N-1)$,
 $q=\max(q',q_M-1)$ and $f_{AC}=0$ for all $AC$ such that $s'<c_N\le s$ or $q'<a_M\le q$.  The representation 
$$
x^Kf=\sum_{c_N=0}^{s}\sum_{a_M=0}^{q}\left(\sum_{\substack{0\le a_i<q_i\\
0\le c_j<\bar{s}_j}}f_{AC}(x)\prod_{i=1}^{M-1}P_i^{a_i}\prod_{j=1}^{N-1}Q_j^{c_j}\right)P_M^{a_M}Q_N^{c_N},
$$
is called an $(M,N)$-expansion of $f$. Notice that if $M$ and $N$ are chosen as in Theorem  \ref{theorem1} then the $(M,N)$-expansion of $f$ coincides with representation (\ref{NUrepresentation}). Moreover, if $M'\ge M$ and $N'\ge N$, where $M$ and $N$ are as in Theorem \ref{theorem1}, then the nonzero terms of the $(M',N')$-expansion of $f$ coincide with the nonzero terms of representation (\ref{NUrepresentation}) for $f$.

We now define the following $\QQ$-valued  maps. The value function $\val$ is defined on the set of monomials in $\{P_i\}_{i\ge 1}\cup\{Q_j\}_{j\ge 1}$ with a coefficient in $k[x]$. Let $M$ and $N$ be nonnegative integers, $\prod_{i=1}^{M}P_i^{b_i}\prod_{j=1}^{N}Q_j^{t_j}$ be a monomial and $f(x)\in k[x]$ be a nonzero polynomial, we set $\val(f\prod_{i=1}^{M}P_i^{b_i}\prod_{j=1}^{N}Q_j^{t_j})=(\ord f)\b_0+\sum_{j=1}^M b_j\b_j+\sum_{j=1}^N t_j\bar{\g}_j$. The following functions are defined on $k[x,y,z]\setminus\{0\}$. Let $f(x,y,z)$ be a nonzero polynomial and $M,\,N$ be positive integers.  Consider  the $(M,N)$-expansion of $f$ 
$$
x^Kf=\sum_{c_N=0}^{s}\sum_{a_M=0}^{q}\left(\sum_{\substack{0\le a_i<q_i\\
0\le c_j<\bar{s}_j}}f_{AC}(x)\prod_{i=1}^{M-1}P_i^{a_i}\prod_{j=1}^{N-1}Q_j^{c_j}\right)P_M^{a_M}Q_N^{c_N},
$$
and set $\n_{M,N}(f)=-K\b_0+\min\{\val(f_{AC}\prod_{i=1}^{M}P_j^{a_j}\prod_{j=1}^{N}Q_j^{c_j})\mid f_{AC}\neq 0\}$.  Also consider representation (\ref{NUrepresentation}) of $f$ 
$$
x^Kf=\sum_{\substack{0\le a_i<q_i\\
0\le c_j<\bar{s}_j}}f_{AC}(x)\prod_{i=1}^MP_i^{a_i}\prod_{j=1}^NQ_j^{c_j},
$$
and set $\n(f)=-K\b_0+\min\{\val(f_{AC}\prod_{i=1}^{M}P_j^{a_j}\prod_{j=1}^{N}Q_j^{c_j})\mid f_{AC}\neq 0\}$. Notice that if $M$ and $N$ are chosen as in Theorem  \ref{theorem1} then $\n(f)=\n_{M,N}(f)$. Moreover, if $M'\ge M$ and $N'\ge N$, where $M$ and $N$ are as in Theorem \ref{theorem1}, then $\n_{M',N'}(f)=\n(f)$. 

We observe basic properties of $\val$, $\n_{M,N}$ and $\n$:
\begin{itemize}
\item[1.] If $H_1,H_2$ are monomials in $\{P_i\}_{i\ge 1}\cup\{Q_j\}_{j\ge 1}$ and $f_1(x),f_2(x)\in k[x]\setminus\{0\}$ then $\val((f_1H_1)(f_2H_2))=\val(f_1H_1)+\val(f_2H_2)$.

\item[2.] If $M,N,L\in\ZZ_{>0}$ and $f\in k[x,y,z]\setminus\{0\}$ then $\n_{M,N}(x^Lf)=L\b_0+\n_{M,N}(f)$.

\item[3.] If $M,N\in\ZZ_{>0}$ and $f,g\in k[x,y,z]\setminus\{0\}$ then $\n_{M,N}(f+g)\ge \min(\n_{M,N}(f),\n_{M,N}(g))$.

\item[4.] If $f,g\in k[x,y,z]\setminus\{0\}$ then $\n(f+g)\ge \min(\n(f),\n(g))$.
\end{itemize}

The next two lemmas will allow us to show that $\n_{M,N}(fg)=\n_{M,N}(f)+\n_{M,N}(g)$.

\begin{lemma}\label{induction base}
Suppose that  $M\in\ZZ_{>0}$ and $g=g_0(x)\prod_{i=1}^{M}P_i^{b_i}$ then $\n_{M,1}(g)=\val(g)$. Moreover, there exists a unique $M$-tuple $(a_1,\dots,a_M)$ in the $P_M$-expansion of $g$ such that $\val(g_{a_1,\dots,a_M}\prod_{j=1}^{M}P_j^{a_j})=\n_{M,1}(g)$. Also, for the $M$-tuple above $a_M=b_M$.
\end{lemma}

\begin{proof} In order to find $\n_{M,1}(g)$ we construct the $(M,1)$-expansion of $g$, which in this case coincides with the $P_M$-expansion. We use induction on $(M,D)$, where $D=\deg_y (\prod_{i=1}^{M-1}P_i^{b_i})$. 

If $M=1$ then $g=g_0(x)P_1^{b_1}$ is the $P_M$-expansion of $g$ and the statement follows.  Assume that $M>1$. Let $g'=g_0(x)\prod_{i=1}^{M-1}P_i^{b_i}$. Then by the inductive hypothesis the $P_{M-1}$-expansion 
$$
g'=\sum_{a_{M-1}=0}^q\left(\sum_{0\le a_i<q_i}g'_{a_1,\dots,a_{M-1}}(x)\prod_{i=1}^{M-2}P_i^{a_i}\right)P_{M-1}^{a_{M-1}}.
$$ 
has only one term of value $\val(g')$ and all other terms are of greater value. Consider the following expansion 
$$
g=\sum_{a_{M-1}=0}^q\left(\sum_{0\le a_i<q_i}g'_{a_1,\dots,a_{M-1}}(x)\prod_{i=1}^{M-2}P_i^{a_i}\right)P_{M-1}^{a_{M-1}}P_M^{b_M}.
$$
Only one term of this expansion has value $\val(g')+b_M\b_M=\val(g)$, all other terms have greater value. Thus proving the statement of the lemma for every term of this expansion will prove the statement for $g$.

We fix a nonzero term $\bar{g}=g'_{a_1,\dots,a_{M-1}}(x)\prod_{i=1}^{M-1}P_i^{a_i}P_M^{b_M}$ in the expansion of $g$ and show that the lemma holds for $\bar{g}$. If $a_{M-1}<q_{M-1}$ then the representation above is the $P_M$-expansion of $\bar{g}$ and the lemma holds for $\bar{g}$. Assume that $a_{M-1}\ge q_{M-1}$, then
$$
\bar{g}=g'_{a_1,\dots,a_{M-1}}(x)\prod_{i=1}^{M-2}P_i^{a_i}(\l_{M-1}x^{n_{M-1,0}}\prod_{i=1}^{M-2}P_i^{n_{M-1,i}}+P_M)P_{M-1}^{a_{M-1}-q_{M-1}}P_M^{b_M}=g_1+g_2,
$$
where $g_1=(\l_{M-1}x^{n_{M-1,0}}g'_{a_1,\dots,a_{M-1}}(x))\prod_{i=1}^{M-2}P_i^{a_i+n_{M-1,i}}P_{M-1}^{a_{M-1}-q_{M-1}}P_M^{b_M}$ and $
g_2=g'_{a_1,\dots,a_{M-1}}(x)\prod_{i=1}^{M-2}P_i^{a_i}P_{M-1}^{a_{M-1}-q_{M-1}}P_M^{b_M+1}$. 

Notice that $\val(g_1)=\val(\bar{g})$ and $\val(g_2)>\val(\bar{g})$. Also notice that by Corollary \ref{new lemma} we have $\deg_y (\prod_{i=1}^{M-1}P_i^{a_i})\le \deg_y g'$, that is   $\deg_y (\prod_{i=1}^{M-1}P_i^{a_i})\le D$.

Since $\deg_y(\prod_{i=1}^{M-2}P_i^{a_i}P_{M-1}^{a_{M-1}-q_{M-1}})<\deg_y (\prod_{i=1}^{M-1}P_i^{a_i})$ by the inductive hypothesis applied to $g_2$ we get $\n_{M,1}(g_2)=\val(g_2)>\val(\bar{g})$ and all terms in the $P_M$-expansion of $g_2$ have values greater than $\val(\bar{g})$. In order to apply the inductive hypothesis to $g_1$ we notice that
$$
\deg_y (\prod_{i=1}^{M-2}P_i^{a_i+n_{M-1,i}}P_{M-1}^{a_{M-1}-q_{M-1}}) <\deg_y(\prod_{i=1}^{M-2}P_i^{a_i}P_{M-1}^{a_{M-1}-q_{M-1}+1})\le \deg_y(\prod_{i=1}^{M-1}P_i^{a_i})
$$
Thus  $\n_{M,1}(g_1)=\val(g_1)=\val(\bar{g})$ and only one term in the $P_M$-expansion of $g_1$ has value $\val(\bar{g})$. Moreover, the power of $P_M$ in this term coincides with the power of $P_M$ in $g_1$. This shows that the $P_M$-expansion of $\bar{g}$ will have a unique term of value $\val(\bar{g})$ and the power of $P_M$ in this term will be $b_M$. All other terms in the $P_M$-expansion of $\bar{g}$ will have greater values, in particular, $\n_{M,1}(\bar{g})=\val(\bar{g})$.  
\end{proof}

\begin{lemma}\label{key}
Suppose that  $M,N\in\ZZ_{>0}$ are such that $\bar{m}_{N-1}< M$. 
\newline Suppose that $f=f_0(x)\prod_{i=1}^{M}P_i^{b_i}\prod_{j=1}^{N}Q_j^{t_j}$ then $\n_{M,N}(f)=\val(f)$.  Moreover, there exists a unique  $AC=(a_1,\dots,a_M,c_1\dots,c_N)$ in the $(M,N)$-expansion of $f$ such that $-K\b_0+\val(f_{AC}\prod_{j=1}^{M}P_j^{a_j}\prod_{j=1}^{N}Q_j^{c_j})=\n_{M,N}(f)$. Also, for $AC$ as above $a_M=b_M$ and $c_N=t_N$.
\end{lemma}

\begin{proof} In order to find $\n_{M,N}(f)$ we construct the $(M,N)$-expansion of $f$. We argue by induction on $(N,D)$, where $D=\deg_z(\prod_{j=1}^{N-1}Q_j^{t_j})$. The base case of $N=1$ follows easily from Lemma \ref{induction base}. 
Assume that $N>1$.

First notice that if the statement of the lemma holds for $x^Lf$ for some $L\in\ZZ_{\ge 0}$ then it also holds for $f$. Let $f'=f_0(x)\prod_{i=1}^{M}P_i^{b_i}\prod_{j=1}^{N-1}Q_j^{t_j}$ and
\begin{equation}\label{f'}
x^{K'}f'=\sum_{a_M=0}^q\sum_{c_{N-1}=0}^s\left(\sum_{\substack{0\le a_i<q_i\\
0\le c_j<\bar{s}_j}}f'_{AC}(x)\prod_{i=1}^{M-1}P_i^{a_i}\prod_{j=1}^{N-2}Q_j^{c_j}\right)P_M^{a_M}Q_{N-1}^{c_{N-1}}
\end{equation}
be the $(M,N-1)$-expansion of $f'$. Then after possibly multiplying $f$ by some power of $x$ we may assume that $K'=0$ and $f'_{AC}\in x^{\bar{r}_{N-1,0}}k[x]$ for all terms of (\ref{f'}). 

By the inductive hypothesis expansion (\ref{f'}) has only one term of value $\val(f')$ and all other terms are of greater value. Moreover, the power of $P_M$ in this unique term of minimal value is $b_M$. Consider the following expansion 
$$
f=\sum_{a_M=0}^q\sum_{c_{N-1}=0}^s\left(\sum_{\substack{0\le a_i<q_i\\
0\le c_j<\bar{s}_j}}f'_{AC}(x)\prod_{i=1}^{M-1}P_i^{a_i}\prod_{j=1}^{N-2}Q_j^{c_j}\right)P_M^{a_M}Q_{N-1}^{c_{N-1}}Q_N^{t_N}.
$$
Only one term of this expansion has value $\val(f')+t_N\g_N=\val(f)$ and the power of $P_M$ in this term is $b_M$, all other terms have greater value and the power of $Q_N$ is $t_N$ for all terms. Notice that the $(M,N)$-expansion of $f$ can be obtained by adding the $(M,N)$-expansions of all terms on the right. Thus proving the statement of the lemma for every term of this expansion will prove the statement for $f$.

We fix a nonzero term $\bar{f}=f'_{AC}(x)\prod_{i=1}^M P_i^{a_i}\prod_{j=1}^{N-1}Q_j^{c_j}Q_N^{t_N}$ in the expansion of $f$ and show that the lemma holds for $\bar{f}$. If $c_{N-1}<\bar{s}_{N-1}$ then the above representation is the $(M,N)$-expansion of $\bar{f}$, and therefore, the statement holds for $\bar{f}$. 

Assume that $c_{N-1}\ge \bar{s}_{N-1}$. In the above representation of $\bar{f}$ replace $Q_{N-1}^{\bar{s}_{N-1}}$ with 
$x^{-\bar{r}_{N-1,0}}(\bar{\m}_{N-1}x^{\bar{n}_{N-1,0}}\prod_{i=1}^{\bar{m}_{N-1}}P_i^{\bar{n}_{N-1,i}}\prod_{j=1}^{N-2}Q_j^{\bar{l}_{N-1,j}}+Q_N)$.  Denote $f''_{AC}=x^{-\bar{r}_{N-1,0}}f'_{AC}(x)$ and $f'''_{AC}=\bar{\m}_{N-1}x^{\bar{n}_{N-1,0}}f''_{AC}(x)$ to simplify notation. Furthermore,  set
$$
f_1=f'''_{AC}\prod_{i=1}^{M}P_i^{a'_i}\prod_{j=1}^{N-2}Q_j^{c_j+\bar{l}_{N-1,j}}Q_{N-1}^{c_{N-1}-\bar{s}_{N-1}}Q_N^{t_N},
$$
 where $a'_i=a_i+\bar{n}_{N-1,i}$ if $i\le \bar{m}_{N-1}$ and $a'_i=a_i$ if $i>\bar{m}_{N-1}$, and   
 $$
 f_2=f''_{AC}\prod_{i=1}^{M}P_i^{a_i}\prod_{j=1}^{N-2}Q_j^{c_j}Q_{N-1}^{c_{N-1}-\bar{s}_{N-1}}Q_N^{t_N+1}.
$$
Then $f''_{AC},f'''_{AC}\in k[x]$, $\val(f_1)=\val(\bar{f})$, $\val(f_2)>\val(\bar{f})$ and $\bar{f}=f_1+f_2$. Our next step is to apply the inductive hypothesis to $f_1$ and $f_2$. We notice that  by Corollary \ref{new q-lemma} we have $\deg_z(\prod_{j=1}^{N-1}Q_j^{c_j})\le\deg_z f'=D$. 

Let $K_1$ and $K_2$ be the powers of $x$ appearing in the $(M,N)$-expansions for $f_1$ and $f_2$, respectively. We set $K=\max(K_1,K_2)$, so that the powers of $x$ appearing in the $(M,N)$-expansions for $x^Kf_1$ and $x^Kf_2$ are both 0.

Since $\deg_{z}(\prod_{j=1}^{N-2}Q_j^{c_j}Q_{N-1}^{c_{N-1}-\bar{s}_{N-1}})<\deg_z(\prod_{j=1}^{N-1}Q_j^{c_j})$ it follows that $\n_{M,N}(x^Kf_2)=\val(x^Kf_2)>\val(x^K\bar{f})$ and all terms in the $(M,N)$-expansion of $x^Kf_2$ have values greater than $\val(x^K\bar{f})$. To apply the inductive hypothesis to $x^Kf_1$ we notice that 
$$
\deg_z(\prod_{j=1}^{N-2}Q_j^{c_j+\bar{l}_{N-1,j}}Q_{N-1}^{c_{N-1}-\bar{s}_{N-1}}) <\deg_z(\prod_{j=1}^{N-2}Q_j^{c_j}Q_{N-1}^{c_{N-1}-s_{N-1}+1})\le\deg_z(\prod_{j=1}^{N-1}Q_j^{c_j})
$$
Thus $\n_{M,N}(x^Kf_1)=\val(x^Kf_1)=\val(x^K\bar{f})$ and only one term in the $(M,N)$-expansion of $x^Kf_1$ has value $\val(x^K\bar{f})$. Moreover, the powers of $P_M$ and $Q_N$ in this unique term coincide with the powers of $P_M$ and $Q_N$ in $f_1$. This shows that the $(M,N)$-expansion of $x^K\bar{f}$ has a unique term of value $\val(x^K\bar{f})$ and the powers of $P_M$ and $Q_N$ in this term are $a'_M=a_M$ and $t_N$, respectively. All other terms in the $(M,N)$-expansion of $x^K\bar{f}$ are of greater value, in particular, $\n_{M,N}(x^K\bar{f})=\val(x^K\bar{f})$. Thus the lemma holds for $x^K\bar{f}$, and therefore it also holds for $\bar{f}$.
\end{proof}

\begin{corollary}\label{M,N-valuation-intermidiate}
Suppose that $M,N\in\ZZ_{>0}$ are such that $\bar{m}_{N-1}< M$. 

Suppose that $f,g\in k[x,y,z]\setminus\{0\}$ then $\n_{M,N}(fg)=\n_{M,N}(f)+\n_{M,N}(g)$. 
\end{corollary}

\begin{proof}
We first notice that if the statement is true for $x^Lf$ and $x^Lg$ for some $L\in\ZZ_{\ge 0}$ then it is also true for $f$ and $g$. Consider the $(M,N)$-expansions 
$$x^{K_1}f=\sum_{a_M=0}^{q_f}\sum_{c_N=0}^{s_f}\left(\sum_{\substack{0\le a_i<q_i\\
0\le c_j<\bar{s}_j}}f_{AC}(x)\prod_{i=1}^{M-1}P_i^{a_i}\prod_{j=1}^{N-1}Q_j^{c_j}\right)P_M^{a_M}Q_N^{c_N}
$$ and
$$
x^{K_2}g=\sum_{\bar{a}_M=0}^{q_g}\sum_{\bar{c}_N=0}^{s_g}\left(\sum_{\substack{0\le \bar{a}_i<q_i\\
0\le \bar{c}_j<\bar{s}_j}}g_{\bar{AC}}(x)\prod_{i=1}^{M-1}P_i^{\bar{a}_i}\prod_{j=1}^{N-1}Q_j^{\bar{c}_j}\right)P_M^{\bar{a}_M}Q_N^{\bar{c}_N}.
$$
After possibly multiplying $f$ and $g$ by some power of $x$ we may assume that $K_1=0$, $K_2=0$ and the power of $x$ appearing on the left hand side in the $(M,N)$-expansion of $f_{AC}(x)g_{\bar{AC}}(x)\prod_{i=1}^M P_i^{a_i+\bar{a}_i}\prod_{j=1}^N Q_j^{c_j+\bar{c}_j}$ is 0 for all $AC$ and $\bar{AC}$ as above.

Let $I_f=\{AC\mid \n_{M,N}(f)=\val(f_{AC}\prod_{i=1}^{M}P_i^{a_i}\prod_{j=1}^{N}Q_j^{c_j})\}$ and $AC_f\in I_f$ be such that $((a_M)_f,(c_N)_f)=\min\{(a_M,c_N)\mid AC\in I_f\}$. We claim that $((a_M)_f,(c_N)_f)<(a_M,c_N)$ for any $AC\in I_f\setminus\{AC_f\}$. Indeed, assume that $AC,AC'\in I_f$ are such that $a_M=a'_M$ and $c_N=c'_N$. Then since
$$\ord(f_{AC})\b_0+\sum_{i=1}^{M-1}a_i\b_i+\sum_{j=1}^{N-1}c_j\bar{\g}_j=\ord(f_{AC'})\b_0+\sum_{i=1}^{M-1}a'_i\b_i+\sum_{j=1}^{N-1}c'_j\bar{\g}_j$$
by Lemma \ref{group representation} we have $a_i=a'_i$  and $c_j=c'_j$ for all $ i, j$ in the range. Thus $AC=AC'$. 

Similarly let $I_g=\{\bar{AC}\mid \n_{M,N}(g)=\val(g_{\bar{AC}}\prod_{i=1}^{M}P_i^{\bar{a}_i}\prod_{j=1}^{N}Q_j^{\bar{c}_j})\}$ and $\bar{AC}_g\in I_g$ be such that $((\bar{a}_M)_g,(\bar{c}_N)_g)=\min\{(\bar{a}_M,\bar{c}_N)\mid \bar{AC}\in I_g\}$. Then $((\bar{a}_M)_g,(\bar{c}_N)_g)<(\bar{a}_M,\bar{c}_N)$ for any $\bar{AC}\in I_g\setminus\{\bar{AC}_g\}$.

Notice that $fg=\sum_{AC,\bar{AC}}\left(f_{AC}(x)g_{\bar{AC}}(x)\prod_{i=1}^M P_i^{a_i+\bar{a}_i}\prod_{j=1}^N Q_j^{c_j+\bar{c}_j}\right)$ and the $(M,N)$-expansion of $fg$ can be obtained by adding the $(M,N)$-expansions of all terms on the right.  Let $h=f_{AC_f}(x)g_{\bar{AC}_g}(x)\prod_{i=1}^M P_i^{(a_i)_f+(\bar{a}_i)_g}\prod_{j=1}^N Q_j^{(c_j)_f+(\bar{c}_j)_g}$ and notice that $\val(h)=\val(f_{AC_f}(x)\prod_{i=1}^M P_i^{(a_i)_f}\prod_{j=1}^N Q_j^{(c_j)_f})+\val(g_{\bar{AC}_g}(x)\prod_{i=1}^M P_i^{(\bar{a}_i)_g}\prod_{j=1}^N Q_j^{(\bar{c}_j)_g})=\n_{M,N}(f)+\n_{M,N}(g)$.  

By Lemma \ref{key} the $(M,N)$-expansion of $h$ has a unique term of value $\n_{M,N}(f)+\n_{M,N}(g)$ and, moreover, the powers of $P_M$ and $Q_N$ in this term are $(a_M)_f+(\bar{a}_M)_g$ and $(c_N)_f+(\bar{c}_N)_g$, respectively. We claim that if $AC\neq AC_f$ or $\bar{AC}\neq \bar{AC}_g$ then 
the $(M,N)$-expansion of $f_{AC}g_{\bar{AC}}\prod_{i=1}^M P_i^{a_i+\bar{a}_i}\prod_{j=1}^N Q_j^{c_j+\bar{c}_j}$ does not have a term of value $\val(h)$ such that its powers of $P_M$ and $Q_N$ are $(a_M)_f+(\bar{a}_M)_g$ and $(c_N)_f+(\bar{c}_N)_g$, respectively. Indeed, set $h'=f_{AC}(x)g_{\bar{AC}}(x)\prod_{i=1}^M P_i^{a_i+\bar{a}_i}\prod_{j=1}^N Q_j^{c_j+\bar{c}_j}$. If $AC\notin I_f$ or $\bar{AC}_g\notin I_G$ then $\val(h')>\val(h)$ and, therefore, all terms in the $(M,N)$-expansion of $h'$ have values greater than $\val(h)$. If $AC\in I_f$ and $\bar{AC}\in I_g$ then $\val(h')=\val(h)$ and the only term of value $\val(h)$ in the $(M,N)$-expansion of $h'$ has powers of $P_M$ and $Q_N$ equal to $a_M+\bar{a}_M$ and $c_N+\bar{c}_N$, respectively. Then if $AC\neq AC_f$ or $\bar{AC}\neq\bar{AC}_g$ we have $((a_M)_f+(\bar{a}_M)_g,(c_N)_f+(\bar{c}_N)_g)<(a_M+\bar{a}_M,c_N+\bar{c}_N)$. This shows that the unique term of minimal value $\val(h)$ from the $(M,N)$-expansion of $h$ will not cancel  in the $(M,N)$-expansion of $fg$. Moreover, all other terms in the $(M,N)$-expansion of $fg$ will be of value greater than or equal to $\val(h)$. Thus $\n_{M,N}(fg)=\val(h)=\n_{M,N}(f)+\n_{M,N}(g)$.
\end{proof}

\begin{corollary}\label{valuation-intermidiate}
If $f,g\in k[x,y,z]\setminus\{0\}$ then $\n(fg)=\n(f)+\n(g)$. 
\end{corollary}
\begin{proof}
Let $M,N\in\ZZ_{>0}$ be such that $\n(f)=\n_{M,N}(f)$, $\n(g)=\n_{M,N}(g)$, $\n(fg)=\n_{M,N}(fig)$ and $M>\bar{m}_{N-1}$. Then the statement follows from Corollary \ref{M,N-valuation-intermidiate}.
\end{proof}

We now extend $\n$ to $k(x,y,z)\setminus \{0\}$. If $f\in k(x,y,z)\setminus\{0\}$ and $f_1,f_2\in k[x,y,z]\setminus\{0\}$ are such that $f=f_1/f_2$ then $\n(f)=\n(f_1)-\n(f_2)$. Due to Corollary \ref{valuation-intermidiate} $\n(f)$ does not depend on representation $f=f_1/f_2$. Thus $\n$ is well defined on $k(x,y,z)\setminus\{0\}$. Moreover, if $f,g\in k(x,y,z)\setminus\{0\}$ then $\n(f+g)\ge\min(\n(f),\n(g))$ and $\n(fg)=\n(f)\n(g)$. Thus $\n$ is a valuation on $k(x,y,z)$.

\begin{lemma}\label{dominate}
Suppose that $f\in k[x,y,z]_{(x,y,z)}$ is not equal to 0.  Then $\n(f)\ge 0$ and $\n(f)>0$ if and only if $f$ lies in the maximal ideal of $k[x,y,z]_{(x,y,z)}$.
\end{lemma}

\begin{proof}
Assume that $f\in k[x,y,z]$ then $f=e_1f_1+e_2f_2+\dots+e_nf_n$, where for each $i$ we have $f_i=x^{a_i}y^{b_i}z^{c_i}$ is a monomial in $x,y,z$; $e_i\in k\setminus\{0\}$ and $(a_i,b_i,c_i)\neq (a_j,b_j,c_j)$ if $i\neq j$.  Then $\n(f)\ge\min(\n(e_if_i))=\min(\val(e_if_i))=\min(a_i\b_0+b_i\b_1+c_i\bar{\g}_i)\ge 0$. In particular, if $f\in(x,y,z)k[x,y,z]$ then $a_i+b_i+c_i>0$ for all $i$ and $\n(f)>0$. Finally, if $f\notin (x,y,z)k[x,y,z]$ then $a_i=b_i=c_i=0$ for some $i$ and since $\n(f-e_if_i)>0$ we have $\n(f)=\min(\n(e_if_i),\n(f-e_if_i))=0$.
\end{proof}

As a consequence of Lemma \ref{dominate} we get that $\n:k(x,y,z)\ra\QQ$ is a valuation on $k(x,y,z)$ dominating $k[x,y,z]_{(x,y,z)}$ with $\n(P_i)=\b_i$ and $\n(Q_i)=\bar{\g}_i$ for all $i$. If $\n':k(x,y,z)\ra\QQ$ is another valuation dominating $k[x,y,z]_{(x,y,z)}$ such that $\n'(P_i)=\b_i$ and $\n'(Q_i)=\bar{\g}_i$ for all $i$ then we check that $\n'(f)=\n(f)$ for any $f\in k[x,y,z]\setminus\{0\}$. To this end we fix $f\in k[x,y,z]\setminus\{0\}$, consider its representation (\ref{NUrepresentation})  
$$
x^Kf=\sum_{\substack{0\le a_i<q_i\\
0\le c_j<\bar{s}_j}}f_{AC}(x)\prod_{i=1}^MP_i^{a_i}\prod_{j=1}^NQ_j^{c_j}. 
$$
and evaluate $\n'$ of every term on the right. Since $\n'$ dominates $k[x,y,z]_{(x,y,z)}$ we have $\n'(a)=0$ for all $a\in k\setminus\{0\}$ and $\n'(f_{AC}(x))=\ord(f_{AC})\b_0$ provided $f_{AC}(x)\neq 0$. Thus, if $f_{AC}(x)\neq 0$ then
\begin{align*}
\n'(f_{AC}(x)\prod_{i=1}^MP_i^{a_i}\prod_{j=1}^NQ_j^{c_j}) &=\ord(f_{AC})\b_0+\sum_{j=1}^Ma_j\b_j+\sum_{j=1}^Nc_j\bar{\g}_j\\
& =\val(f_{AC}(x)\prod_{i=1}^MP_i^{a_i}\prod_{j=1}^NQ_j^{c_j}),
\end{align*} 
By Corollary \ref{order1} there exist a unique term $AC'=(a'_1,\dots,a'_M,c'_1,\dots,c'_N)$ such that 
\begin{align*}
\ord(f'_{AC})\b_0+\sum_{j=1}^M & a'_j\b_j+\sum_{j=1}^Nc'_j\bar{\g}_j =\\
&\min\{\ord(f_{AC})\b_0+\sum_{j=1}^Ma_j\b_j+\sum_{j=1}^Nc_j\bar{\g}_j\mid f_{AC}(x)\neq 0\}.
\end{align*}
Thus
 $\n'(f)=-K\b_0+\n'(f_{AC'}(x)\prod_{i=1}^MP_i^{a'_i}\prod_{j=1}^NQ_j^{c'_j})=\n(f)$. This completes the proof of Theorem \ref{main}.

\section{main example}\label{main_example}

In this section we provide an example of  valuation centered in $k[x,y,z]_{(x,y,z)}$. We use defining polynomials to construct a valuation and then work with the sequence of jumping polynomials to describe its value semigroup. In this example only one member of the sequence $\{\bar{r}_{i,0}\}_{i\ge 0}$ is greater than 0.

\begin{example} \label{main_ex}

Let $\b_0=1$, $\b_1=1\frac{1}{2}$, $\b_2=3\frac{1}{4}$, $\b_3=6\frac{5}{8},\dots$ and $\bar{\g}_1=2\frac{1}{4}$, $\bar{\g}_2=4\frac{1}{3}$, $\bar{\g}_3=13\frac{1}{9}$, $\bar{\g}_4=39\frac{10}{27}\dots$, where $\b_i=2\b_{i-1}+\frac{1}{2^i}$ and $\bar{\g}_{i+1}=3\bar{\g}_i+\frac{1}{3^i}$ for all $i>1$. Then for all $i\ge 1$ in notation of section \ref{numerical data} we have 
$$
G_i=\frac{1}{2^i}\ZZ\quad\text{and}\quad\bar{H}_{i+1}=\frac{1}{4\cdot3^i}\ZZ,\quad \text{while}\quad G_0=\ZZ\quad\text{and}\quad\bar{H}_1=\frac{9}{4}\ZZ.
$$
So $q_i=2$ for all $i>0$ and $\bar{s}_1=1$, $\bar{m}_1=2$,  and  $\bar{s}_i=3$, $\bar{m}_i=2$ for all $i>1$. Since 
$\bar{\g}_1=-\b_0+\b_2$ we have $\bar{r}_{1,0}=1$.  Also notice that $2\b_1=3\b_0$ and $2\b_i=5\cdot2^{i-2}\b_0+\b_{i-1}$ for all $i>1$, and $3\bar{\g}_2=13\b_0$ and $3\bar{\g}_i=35
 \cdot3^{i-3}\b_0+\bar{\g}_{i-1}$ for all $i>2$. In particular, this shows $\bar{r}_{i,0}=0$  for all $i\ge 2$.  

Let $\l_i=\bar{\m}_i=1$ for all $i>0$. Since the inequalities $\b_{i+1}>q_i\b_i$ and $\bar{\g}_{i+1}>\bar{r}_{i,0}\b_0+\bar{s}_i\bar{\g}_i$ are satisfied for all $i>0$ the set of polynomials $\{P_i\}_{i\ge 0}\cup\{Q_i\}_{i>0}$ as constructed in section \ref{numerical data} defines a valuation $\n$ on $k(x,y,z)$.

We have $P_0=x$, $P_1=y$, $P_2=y^2-x^3$ and  $Q_1=z$, $Q_2=xz-P_2$, $Q_3=Q_2^3-x^{13}$. The recursive formulas for $P_i$ and $Q_{i+1}$ when $i>2$ are
\begin{align*}
P_{i} &=P_{i-1}^2-x^{5\cdot 2^{i-3}}P_{i-2}\\
Q_{i+1}& =Q_i^3-x^{35\cdot 3^{i-3}}Q_{i-1}.
\end{align*}

We will construct several first members of the subsequence $\{T_i\}_{i>0}$ of jumping polynomials to understand the pattern for nonredundant jumping polynomials. We use $(M,N)$-expansions to find the required values and residues. 

Since $T_1=z$ we get $\g_1=2\frac{1}{4}$, $s_1=1$, $m_1=2$, $\D_1=\{(1,0,0,1),(0,0,0,2),(0,0,1,1)\}$ and $\d_1=3$. Since   $xT_1=P_2+Q_2$ we have 
$$ T_1^2=x^3P_1+x^{-2}(P_3+2P_2Q_2+Q_2^2),\quad P_2T_1=x^4P_1+x^{-1}(P_3+P_2Q_2).$$ 
Thus the immediate successors of $T_1$ are
$$
\begin{array}{ll}
T_2 =xT_1-P_2=Q_2,\quad &\g_2=\n(T_2)=4\frac{1}{3}\\
T_3 =T_1^2-x^3P_1=x^{-2}(P_3+2P_2Q_2+Q_2^2),\quad & \g_3=\n(T_3)=4\frac{5}{8}\\
T_4 =P_2T_1-x^4P_1=x^{-1}(P_3+P_2Q_2),\quad &  \g_4=\n(T_4)=5\frac{5}{8}\\
\end{array}
$$
We notice that $\g_4$ is a dependent value. Moreover, $T_4=xT_3-T_1T_2$. From further computations it will follow that $xT_3$ and $T_1T_2$ are both irreducible with respect to $\T$, and therefore, $T_4$ is redundant by Lemma \ref{redundant_suf}. 

Consider $T_2$ and $\g_2=4\frac{1}{3}$. We have $s_2=3$, $m_2=2$, $r_{2,0}=0$, $\D_2=\{(0,0,0,0,3)\}$ and $\d_2=1$. The only immediate successor of $T_2$ is
$$
T_5 =T_2^3-x^{13}=Q_3,\quad\quad\quad\quad\quad\quad\quad\quad \g_5=\n(T_5)=13\frac{1}{9}\\
$$

Consider $T_3$ and $\g_3=4\frac{5}{8}$. We have $s_3=1$, $m_3=3$, $r_{3,0}=2$, $\d_3=4$ and $\D_3=\{(2,0,0,0,0,0,1),(1,0,1,0,0,0,1),(0,0,0,0,0,0,2),(0,0,0,1,0,0,1)\}$.  Since  $x^2T_3=P_3+2P_2Q_2+Q_2^2$   we get
$$T_3^2=x^6P_2+x^{-4}(P_4+4P_2P_3Q_2+hvt),\quad P_3T_3=x^8P_2+x^{-2}(P_4+2P_2P_3Q_2 +hvt).$$
Also 
$$xP_2T_3=x^{-1}(P_2P_3+2x^5P_1Q_2+2P_3Q_2+hvt)\quad  \text{ and } \quad P_3T_1=x^{-1}(P_2P_3+P_3Q_2).$$
Thus the immediate successors of $T_3$ are
$$
\begin{array}{ll}
T_6 =x^2T_3-P_3=2P_2Q_2+Q_2^2,\quad & \g_6=\n(T_6)=7\frac{7}{12}\\
T_7 =xP_2T_3-P_3T_1=2x^4P_1Q_2+x^{-1}(P_3Q_2+hvt),\quad &  \g_7=\n(T_7)=9\frac{1}{4}\\
T_8 =T_3^2-x^6P_2=x^{-4}(P_4+4P_2P_3Q_2+hvt),\quad & \g_8=\n(T_8)=9\frac{5}{16}\\
T_9 =P_3T_3-x^8P_2=x^{-2}(P_4+2P_2P_3Q_2+hvt),\quad &  \g_9=\n(T_9)=10\frac{5}{16}\\
\end{array}
$$
We notice that $\g_6, \g_7,\g_9$ are dependent values. Moreover, $T_6, T_7, T_9$ are redundant jumping polynomials since $T_6=2P_2T_2+T_2^2$, $T_7=2x^4P_1T_2+xT_2T_3-T_1T_2^2$ and $T_9=x^2T_8-2P_2T_2T_3-T_2^2T_3$.

It appears that the following sequence of polynomials $\{R_i\}_{i>0}$ is of interest: $R_1=T_2=Q_1$, $R_2=R_1^2-x^3P_1$ and $R_i=R_{i-1}^2-x^{3\cdot 2^{i-2}}P_{i-1}$ for all $i>2$. 

\begin{conjecture}
Suppose that $j\in\ZZ_{\ge 0}$ is such that $T_j$ is a nonredundant jumping polynomial. Then there exists $i\in\ZZ_{\ge 0}$ such that $T_j=Q_i$ or $T_j=R_i$. 
\end{conjecture}

We will not provide a proof for the conjecture due to its length and technicality. Instead we will notice the following implication of the conjecture:  the value semigroup $\n(k[x,y,z]_{(x,y,z)})$ is generated by the set of values $\{\n(P_i)\}_{i\ge 0}\cup\{\n(Q_i)\}_{i\ge 1}\cup\{\n(R_i)\}_{i\ge 1}$. This weaker statement will be the main statement of Example \ref{main_ex}. To prove it we develop some terminology and look at the properties  of the sequence $\{R_i\}_{i\ge 1}$.

We say that $f\in k(x,y,z)$ is an admissible monomial in $(P_m,Q_n)$ if $f$ can be written as $\l\prod_{i=0}^m P_i^{a_i}\prod_{j=2}^n Q_j^{b_j}$, where $\l\in k$, $a_0\in\ZZ_{\ge 0}$, $a_i\in\{0,1\}$ for all $i\ge 1$ and $b_j\in\{0,1,2\}$ for all $j\ge 2$. Denote the set of all admissible monomials in $(P_m,Q_n)$ by $Mon(P_m, Q_n)$. We say that $f\in k(x,y,z)$ is an admissible monomial in $(P_m,Q)$ if there exists $n$ such that $f\in Mon(P_m, Q_n)$.  Denote the set of all admissible monomials in $(P_m,Q)$ by $Mon(P_m, Q)$.

 We say that $f\in k(x,y,z)$ is an admissible monomial in $(Q_n, R_m)$ if $f$ can be written as $\l x^ay^b\prod_{j=2}^n Q_j^{b_j}\prod_{i=1}^m R_i^{c_i}$, where $\l\in k$, $a\in\ZZ_{\ge 0}$, $b_j\in\{0,1,2\}$ for all $j\ge 2$ and $b, c_i\in\{0,1\}$ for all $i\ge 1$. Denote the set of all admissible monomials in $(Q_n, R_m)$ by $Mon(Q_n,R_m)$. We say that $f\in k(x,y,z)$ is an admissible monomial in $(Q, R_m)$ if there exists $n$ such that $f\in Mon( Q_n, R_m)$.  Denote the set of all admissible monomials in $(Q,R_m)$ by $Mon(Q, R_m)$. Finally, we say that $f\in k(x,y,z)$ is an admissible monomial in $(Q, R)$ if there exists $m$ such that $f\in Mon( Q, R_m)$.  Denote the set of all admissible monomials in $(Q,R)$ by $Mon(Q, R)$.

 We say that $f\in k(x,y,z)$ is an admissible polynomial in $(P_m,Q_n)$ (or $(P_m,Q)$,  or $(Q_n,R_m)$, or $(Q,R_m)$, or $(Q,R)$) if $f$ can be written as a sum of admissible monomials in $(P_m,Q_n)$ (or $(P_m,Q)$, or $(Q_n,R_m)$, or $(Q,R_m)$, or $(Q,R)$, respectively). Denote the set of all admissible polynomials in  $(P_m,Q_n)$ (or $(P_m,Q)$, or $(Q_n,R_m)$, or $(Q,R_m)$, or $(Q,R)$)  by $Poly(P_m, Q_n)$ (or $Poly(P_m,Q)$, or $Poly(Q_n,R_m)$, or $Poly(Q,R_m)$, or $Poly(Q,R)$, respectively).

\begin{lemma}\label{ExL1}
Suppose that  $f,g\in Poly(P_m,Q_1)$ then $fg\in Poly(P_{m+1},Q_1)$.
\end{lemma}
\begin{proof}
We notice that if $f,g\in Poly(P_m,Q_1)$ then $f,g\in k[x,y]$ and $\deg_y f<2^m$, $\deg_y g<2^m$. Thus $(fg)\in k[x,y]$ and $\deg_y (fg)<2^{m+1}$.  Then the statement follows from Lemma \ref{lemma1}. 
\end{proof}

\begin{lemma}\label{ExL2}
Suppose that  $n\ge 2$ and $f,g\in Mon(P_0,Q_n)$ then $fg=h_1+h_2Q_{n+1}$ where $h_1,h_2\in Poly(P_0,Q_n)$.
\end{lemma}
\begin{proof}
We use induction on $n$. If  $n=2$ then $f=\l x^aQ_2^b$ and $g=\m x^rQ_2^t$, where $a,r\in\ZZ_{\ge 0}$ and $b,t\in\{0,1,2\}$.  If $b+t<3$ then $fg=\l\m x^{a+r}Q_2^{b+t}\in Mon(P_0, Q_2)$. If $b+t\ge 3$ then write $Q_2^{b+t}=(x^{13}+Q_3)Q_2^{b+t-3}$ to get
$$
fg=\l\m x^{a+r+13}Q_2^{b+t-3}+\l\m x^{a+r}Q_2^{b+t-3}Q_3=h_1+h_2Q_3,
$$
where $h_1,h_2\in Mon(P_0,Q_2)$.

Assume $n>2$ then $f=f'Q_n^b$ and $g=g'Q_n^t$, where $f',g'\in Mon(P_0,Q_{n-1})$ and $b,t\in\{0,1,2\}$. By the inductive hypothesis we have 
$f'g'=h'_1+h'_2Q_n$, where $h'_1,h'_2\in Poly(P_0, Q_{n-1})$. If $b+t<2$ then $fg=h'_1Q_n^{b+t}+h'_2Q_n^{b+t+1}\in Poly(P_0,Q_n)$. We write $Q_n^3=x^{35\cdot 3^{n-3}}Q_{n-1}+Q_{n+1}$ to get the required representation when $b+t\ge 2$.

Assume first that $b+t=2$ then
$
fg=h_1'Q_n^2+h'_2(x^{35\cdot 3^{n-3}}Q_{n-1})+h_2'Q_{n+1}.
$
Let $e\in Mon(P_0,Q_{n-1})$ be an admissible monomial in the representation of $h'_2$. By the inductive hypothesis we have
$$
e(x^{35\cdot 3^{n-3}}Q_{n-1})=e_1+e_2Q_n,
$$
where $e_1,e_2\in Poly(P_0,Q_{n-1})$. This shows that $h'_2(x^{35\cdot 3^{n-3}}Q_{n-1})=\bar{h}_1+\bar{h}_2Q_n$ for some $\bar{h}_1,\bar{h}_2\in Poly(P_0,Q_{n-1})$. Thus 
$$
fg=\bar{h}_1+\bar{h}_2Q_n+h'_1Q_n^2+h'_2Q_{n+1}=h_1+h_2Q_{n+1},
$$
where $h_1,h_2\in Poly(P_0,Q_n)$. 

Assume now that  $b+t\ge 3$ then 
$$
fg=(h_1'+h'_2Q_n)x^{35\cdot 3^{n-3}}Q_{n-1}Q_n^{b+t-3}+(h'_1Q_n^{b+t-3}+h_2'Q_n^{b+t-2})Q_{n+1}.
$$
Notice that $b+t-2\le 2$. Thus, if $\bar{h}=h'_1Q_n^{b+t-3}+h_2'Q_n^{b+t-2}$ then $\bar{h}\in Poly(P_0,Q_n)$. Also if $e\in Mon(P_0,Q_n)$ is an admissible monomial in the expansion of $h'_1+h'_2Q_n$ then $e=e'Q_n^a$, where $e'\in Mon(P_0,Q_{n-1})$ and  $a \in\{0,1\}$. So, the product of monomials $e'Q_n^a x^{35\cdot 3^{n-3}}Q_{n-1}Q_n^{b+t-3}$ satisfies the condition $b+t-3+a\le 2$. Thus by the argument above we have  
$$
ex^{35\cdot 3^{n-3}}Q_{n-1}Q_n^{b+t-3}=e_1+e_2Q_{n+1},
$$ 
where $e_1,e_2\in Poly(P_0,Q_n)$. This shows that 
$$(h_1'+h'_2Q_n)x^{35\cdot 3^{n-3}}Q_{n-1}Q_n^{b+t-3}=\bar{h}_1+\bar{h}_2Q_{n+1}$$
 for some $\bar{h}_1,\bar{h}_2\in Poly(P_0,Q_n)$. Thus, $
fg=\bar{h}_1+(\bar{h}_2+\bar{h})Q_{n+1}$is the required representation.
\end{proof}

\begin{corollary}\label{ExCor} 
If $f,g\in Mon(P_m,Q_n)$ then $fg\in Poly(P_{m+1},Q_{n+1})$.

If $f,g\in Poly(P_m,Q_n)$ then $fg\in Poly(P_{m+1},Q_{n+1})$.

If $f,g\in Poly(P_m,Q)$ then $fg\in Poly(P_{m+1},Q)$.

If $h_1,h_2,\dots, h_k\in Poly(P_0,Q)$ then $h_1h_2\cdots h_k\in Poly(P_0,Q)$.
\end{corollary}

\begin{proof} Only the first statement is nontrivial. We notice that  if  $f,g\in Mon (P_m, Q_n)$  there exist $f_1,g_1\in Mon(P_m,Q_1)$ and $f_2,g_2\in Mon(P_0, Q_n)$ such that $f=f_1f_2$ and $g=g_1g_2$. Then $f_1g_1\in Poly(P_{m+1},Q_1)$ by Lemma \ref{ExL1} and $f_2g_2\in Poly(P_0,Q_{n+1})$ by Lemma \ref{ExL2}. Let $e_1\in Mon(P_{m+1},Q_1)$ be an admissible monomial in the expansion of $f_1g_1$ and $e_2\in Mon(P_0,Q_{n+1})$ be an admissible monomial in the expansion of $f_2g_2$ then $e_1e_2\in Mon(P_{m+1},Q_{n+1})$. This shows that $fg=f_1g_1f_2g_2\in Poly(P_{m+1},Q_{n+1})$. 
\end{proof}

In the proof of the next statement the following property of the sequence $\{\b_i\}_{i\ge 0}$ is  used: if $i\ge 2$ then $\b_{i+1}=\sum_{j=2}^i\b_j+\b_2+(1/4-1/2^{i+1})$. Indeed, $\b_3=\b_2+\b_2+1/8$ and for $i>2$ we have $\b_{i+1}=2\b_i+1/2^{i+1}=\b_i+\sum_{j=2}^{i-1}\b_j+\b_2+(1/4-1/2^i)+1/2^{i+1}=\sum_{j=2}^i\b_j+\b_2+(1/4-1/2^{i+1})$. In particular, $\b_{i+1}<\sum_{j=2}^i\b_j+3\frac{1}{2}$ for all $i\ge 2$. 

\begin{lemma}\label{PQ-expand}
Suppose that $i\ge 1$ then $x^{2^{i-1}}R_i=P_{i+1}+2^{i-1}Q_2\prod_{j=2}^iP_j+h_i$, where $h_i\in Poly(P_i,Q_i)$ and $\n(h_i)>\n(Q_2\prod_{j=2}^iP_j)>\n(P_{i+1})$. 

In particular, $x^{2^{i-1}}R_i=P_{i+1}+r_i$, where $r_i\in Poly(P_i,Q)$ and $\n(r_i)>\b_{i+1}$, and $\n(R_i)=\b_{i+1}-2^{i-1}\b_0$.
\end{lemma}

\begin{proof}
We use induction on $i$. If $i=1$ then, indeed, $xR_i=P_2+Q_2$ and $\n(Q_2)>\n(P_2)$. Assume that $i>1$ and the statement is true for $i-1$ then 
\begin{align*}
 x^{2^{i-1}}R_i &=x^{2^{i-1}}(R_{i-1}^2-x^{3\cdot 2^{i-2}}P_{i-1})=(P_i+2^{i-2}Q_2\prod_{j=2}^{i-1}P_j+h_{i-1})^2-x^{5\cdot 2^{i-2}}P_{i-1}\\
&=(P_i^2-x^{5\cdot 2^{i-2}}P_{i-1})+2^{i-1}Q_2\prod_{j=2}^iP_j+h_i=P_{i+1}+2^{i-1}Q_2\prod_{j=2}^iP_j+h_i,
\end{align*}
where $h_i=2^{2i-4}Q_2^2\prod_{j=2}^{i-1}P_j^2+2h_{i-1}P_i+2^{i-1}h_{i-1}Q_2\prod_{j=2}^{i-1}P_j+h_{i-1}^2$.

Observe that  $\n(Q_2\prod_{j=2}^iP_j)=\sum_{j=2}^i\b_j+4\frac{1}{3}>\sum_{j=2}^i\b_j+3\frac{1}{2}>\n(P_{i+1})$. Also, since 
\begin{align*}
\n(h_{i-1}^2)>\n(h_{i-1}Q_2\prod_{j=2}^{i-1}P_j)>\n(h_{i-1}P_i) &>\n(Q_2\prod_{j=2}^{i-1} P_jP_i)\\
\text{ and }\quad\quad \n(Q_2^2\prod_{j=2}^{i-1}P_j^2) &>\n(Q_2\prod_{j=2}^{i-1} P_jP_i)
\end{align*} 
we have $\n(h_i)>\n(Q_2\prod_{j=2}^i P_j)$. We notice that $h_{i-1}P_i\in Poly(P_i,Q_{i-1})$ since the product of every admissible monomial in the representation of $h_{i-1}$ and $P_i$ is an admissible monomial in $(P_i, Q_{i-1})$. Finally, $h_{i-1}^2,\, h_{i-1}Q_2\prod_{j=2}^{i-1}P_j,\, Q_2^2\prod_{j=2}^{i-1}P_j^2\in Poly(P_i,Q_i)$ by Corollary \ref{ExCor}. Thus $h_i\in Poly(P_i,Q_i)$.
\end{proof}

The next corollary is a restatement of Corollary \ref{order1}. To align current notation with notation of section \ref{preliminaries} we set $\b'_0=\b_0$, $\b'_1=\b_1$ and $\b'_i=\bar{\g}_i$ for all $i\ge 2$. Also set $\g'_i=\n(R_i)$ for all $i\ge 1$. Then $q'_1=2$ and $q'_i=3$ for all $i\ge 2$ and $s'_i=2$ for all $i\ge 1$. We notice that $x^{a_0}y^{a_1}\prod_{j=2}^n Q_j^{a_j}\prod_{j=1}^m R_j^{c_j}$ is an admissible monomial in $(Q,R)$ if and only if $(a_0,\dots,a_n,c_1,\dots,c_m)$ satisfies the conditions Corollary \ref{order1}. 

\begin{corollary}\label{neworder}
Suppose that $f,g\in Mon(Q,R)$. Let $f=x^{a_0}y^{a_1}\prod_{j=2}^n Q_j^{a_j}\prod_{j=1}^m R_j^{c_j}$ and $g=x^{b_0}y^{b_1}\prod_{j=2}^n Q_j^{b_j}\prod_{j=1}^m R_j^{d_j}$. If $(a_0,\dots,a_n,c_1,\dots,c_m)\neq(b_0,\dots, b_n,d_1,\dots,b_m)$ then $\n(f)\neq\n(g)$.
\end{corollary}

It now follows that for $f\in Poly(Q,R)$ to find $\n(f)$ it is enough to find the minimum of values of admissible monomials in the expansion of $f$. Our next goal is to show that if $f\in k[x,y,z]$ then $f\in Poly(Q,R)$ and to claim that $\n(f)$ belongs to the semigroup generated by $\{\n(x),\n(y)\}\cup\{\n(Q_i)\}_{i\ge 2}\cup\{\n(R_i)\}_{i\ge 1}$.

\begin{lemma}\label{ExL3}
Suppose that $f,g\in Poly(Q,R_m)$ and $h\in Poly(P_{m+1},Q)$. Then  $fg\in Poly(Q,R_{m+1})$ and $h\in Poly(Q,R_m)$.
\end{lemma}

\begin{proof}
It is sufficient to prove the statement under the assumption that $f,g$ and $h$ are admissible monomials. We use induction on $m$. If $m=0$ then $f=y^bf'$ and $g=y^vg'$, where   $f',g'\in Mon(P_0,Q)$ and $b,v\in\{0,1\}$.  If $b+v\le 1$ then $fg=y^{b+v}(f'g')\in Poly(Q,R_0)$, since $f'g'\in Poly(P_0,Q)$ by Corollary \ref{ExCor}. If $b+v=2$ then 
$$
fg=(x^3+P_2)f'g'=(x^3-Q_2)f'g'+xf'g'R_1.
$$
By Corollary \ref{ExCor} we have $(x^3-Q_2)f'g', xf'g'\in Poly(P_0,Q)$ and therefore $xf'g'R_1,fg\in Poly(Q,R_1)$. Also the statement for $h$ holds since $Poly(P_1,Q)=Poly(Q,R_0)$.

Assume that $m\ge 1$. Then $f=f'R_m^c$ and $g=g'R_m^w$, where $f',g'\in Mon(Q,R_{m-1})$ and $c,w\in\{0,1\}$. 
By the inductive hypothesis  we have $f'g'\in Poly(Q,R_m)$.  Let $e=e'R_m^a$, where $e'\in Mon(Q,R_{m-1})$ and $a\in\{0,1\}$, be one of the admissible monomials in the expansion of $f'g'$. 

If $a+c+w\le 1$ then $eR_m^{c+w}\in Mon(Q,R_m)$. If $a+c+w=2$ then 
$$
eR_m^{c+w}=e'(x^{3\cdot 2^{m-1}}P_m+R_{m+1})=x^{3\cdot 2^{m-1}}e'P_m+e'R_{m+1}.
$$
We have  $P_m\in Poly(Q,R_{m-1})$ by the inductive hypothesis, $x^{3\cdot 2^{m-1}}e'\in Mon(Q,R_{m-1})$ and $e'R_{m+1}\in Mon(Q,R_{m+1})$. 
Applying the inductive hypothesis to the product of $P_m$ and $x^{3\cdot 2^{m-1}}e'$ we get $x^{3\cdot 2^{m-1}}e'P_m\in Poly(Q,R_m)$. Thus $eR_m^{c+w}\in Poly(Q,R_{m+1})$.

If $a+c+w=3$ then 
$$
eR_m^{c+w}=e'R_m(x^{3\cdot 2^{m-1}}P_m+R_{m+1})=x^{3\cdot 2^{m-1}}e'P_mR_m+e'R_mR_{m+1}
$$
We have $x^{3\cdot 2^{m-1}}e'P_m\in Poly(Q,R_m)$ by the above argument, and $e'R_mR_{m+1}\in Mon(Q,R_{m+1})$. 
Let $d=d'R_m^b$, where $d'\in Mon(Q,R_{m-1})$ and $b\in\{0,1\}$, be one of the admissible monomials in the expansion of $x^{3\cdot 2^{m-1}}e'P_m$. Then the product $dR_m=d'R_m^bR_m$ satisfies the condition that 
$b+1\le 2$ and therefore, $dR_m\in Poly(Q,R_{m+1})$ as shown above.  This shows that $x^{3\cdot 2^{m-1}}e'P_mR_m\in Poly(Q,R_{m+1})$ and therefore, $eR_m^{c+w}\in Poly(Q,R_{m+1})$. Thus $fg\in Poly(Q,R_{m+1})$.

We now show that $h\in Poly(Q,R_m)$. If $h\in Mon(P_m,Q)$ then $h\in Poly(Q,R_{m-1})$ by the inductive hypothesis. Assume that $h\notin  Mon(P_m,Q)$, then $h=h'P_{m+1}$, where $h'\in Mon(P_m,Q)$. Applying Lemma \ref{ExL3} and the inductive hypothesis we get 
$$
h=x^{2^{m-1}}h'R_m-h'r_m,
$$
where $h',r_m\in Poly(Q,R_{m-1})$ and $x^{2^{m-1}}h'R_m\in Poly(Q,R_m)$. Finally,by the inductive hypothesis $h'r_m\in Poly(Q,R_m)$ and, therefore,  $h\in Poly(Q,R_m)$.
\end{proof}

The following statement follows at once from Lemma \ref{ExL3}
\begin{corollary}\label{product} If $f_1,f_2,\dots f_k\in Poly(Q,R)$ then $f_1f_2\cdots f_k\in Poly(Q,R)$.
\end{corollary}

We can now completely describe the value semigroup of $\n$.
\begin{theorem} $\n(k[x,y,z]_{(x,y,z)})$ is a semigroup generated by $\{1,1\frac{1}{2},2\frac{1}{4},4\frac{5}{8},\dots\}\cup\{4\frac{1}{3},13\frac{1}{9},39\frac{10}{29},\dots\}$.
\end{theorem}
\begin{proof}
Let $S$ be the semigroup generated by the set in the statement of the theorem. It suffices to show that $\n(f)\in S$ for every $f\in k[x,y,z]$. We write $f=f_1+\dots+f_l$, where for all $1\le i\le l$ we have $f_i=\l_ix^{a_i}y^{b_i}z^{c_i}$ with $\l_i\in k$ and $a_i,b_i,c_i\in \ZZ_{\ge 0}$. Since $f_i=(\l_ix^{a_i}y)(y)\dots(y)(R_1)\dots(R_1)$ by Corollary \ref{product} we have $f_i\in Poly(Q,R)$. Therefore, $f\in Poly(Q,R)$. Thus, by Corollary \ref{neworder} $\n(f)\in S$.
\end{proof}
\end{example}


\section{More examples of semigroups of valuations centered in a 3-dimensional regular local ring}\label{examples}

In this section we construct two more examples of valuations centered in $k[x,y,z]_{(x,y,z)}$. We use  defining polynomials to construct a valuation on $k(x,y,z)$ and then consider a sequence of jumping polynomials to understand its value semigroup. In the first example the set $\{\bar{r}_{i,0}|\bar{r}_{i,0}>0\}$ is empty and generators of the value semigroup are the values of defining polynomials. In the second example the set $\{\bar{r}_{i,0}|\bar{r}_{i,0}>0\}$ has two elements. We observe that already in the case of just two $\bar{r}_{i,0}$ greater than zero the pattern for the sequence of generators of the value semigroup becomes quite complicated. 

\begin{example}\label{Example1}
Let $\{\b_i\}_{i\ge 0}$ and $\{\b'_i\}_{i\ge 0}$ be sequences of positive rational numbers such that $\b_0=\b'_0$. Using notation of section \ref{properties1} for all $i\ge 0$ we set
\begin{align*}
S_i=\sum_{j=0}^{i}\b_j\ZZ_{\ge 0},\quad & G_i=\sum_{j=0}^{i}\b_j\ZZ, \;& q_i=\min\{q\in\ZZ_{>0}|q\b_i\in G_{i-1}\}\\
S'_i=\sum_{j=0}^{i}\b'_j\ZZ_{\ge 0},\quad  & G'_i=\sum_{j=0}^{i}\b'_j\ZZ,  & q'_i=\min\{q\in\ZZ_{>0}|q\b'_i\in G'_{i-1}\}
\end{align*} 
We assume that $\gcd(q_i,q'_j)=1$ for all $i,j>0$ and require $\b_{i+1}>q_i\b_i$ and $\b'_{i+1}>q'_i\b'_i$ for all $i>0$.  We also assume that infinitely many $q_i$ and $q'_i$ are greater than 1.

To construct a defining sequence of polynomials, for all $i>0$ we set $\bar{\g}_i=\b'_i$ and fix $\l_i,\bar{\m}_i\in k\setminus\{0\}$. Then for all $i>0$ we have
$$\bar{H}_i+G=(\prod_{j=1}^{\infty}\frac{1}{q_j}\prod_{j=1}^i\frac{1}{q'_j})\b_0\ZZ\quad {\text{ and }}\quad \bar{H}_i+G_0=(\prod_{j=1}^i\frac{1}{q'_j})\b_0\ZZ.$$
 Thus $\bar{s}_i=q'_i$ and $\bar{m}_i=0$ for all $i>0$. By Corollary \ref{positivity1} applied to $\bar{s}_i\bar{\g}_i=q'_i\b'_i$, an element of  $(G_0+\bar{H}_{i-1})=G'_{i-1}$, we get $a_i>0$  and $\bar{n}_{i,0}=a_i$, $\bar{r}_{i,0}=0$ for all $i>0$. We notice that $\bar{s}_i\bar{\g}_i\in(S_0+\bar{U}_{i-1})$  for all $i>0$. Also,  the inequality $\bar{\g}_{i+1}>\bar{r}_{i,0}\b_0+\bar{s}_i\bar{\g}_i$ holds for all $i>0$. Therefore, the set of polynomials $\{P_i\}_{i\ge 0}\cup\{Q_i\}_{i>0}$ as constructed in section \ref{numerical data} defines a valuation $\n$ on $k(x,y,z)$. 

We claim that in this case the sequence of jumping polynomials  $\{P_i\}_{i\ge 0}\cup\{T_i\}_{i>0}$ for $\n$ as defined in section \ref{construction} coinsides with the sequence of defining polynomials $\{P_i\}_{i\ge 0}\cup\{Q_i\}_{i>0}$. We use induction on $i$ to show that $T_i=Q_i$,  $\d_{i-1}=1$, $m_{i-1}=0$ and $\T_i=\P\cup\{Q_j^{q'_j}\}_{j=1}^{i-1}$ for all $i>0$. If $i=1$ the statetemnt holds.

Assume that  $i>0$ and $T_j=Q_j$, $\d_{j-1}=1$, $m_{j-1}=0$ and $\T_j=\P\cup\{Q_j^{q'_e}\}_{e=1}^{j-1}$ for all $0<j\le i$.  Then $\g_j=\bar{\g}_j$ for all $0<j\le i$ and $H_{i-1}=\bar{H}_{i-1}$, $U_{i-1}=\bar{U}_{i-1}$. Thus we have $s_i=\bar{s}_i$, $m_i=\bar{m}_i=0$ and $s_i\g_i\in(S_0+U_{i-1})$. This implies $\d_i=1$, $\D_i=\{(0,0,\dots,0,1)\}$=$\{\bar{d}\}$ and $\T_{i+1}=\T_i\cup\{Q_i^{q'_i}\}=\P\cup\{Q_j^{q'_j}\}_{j=1}^i$. Also, the only immediate successor of $T_i$ is $T_{i+1}$ since $\sum_{j=0}^{i}\d_j=i+1$. 

Finally, to show that $T_{i+1}=Q_{i+1}$ we notice that if $\bar{n}_{i,0},\bar{l}_{i,1},\dots,\bar{l}_{i,i-1}$ are as defined in the construction of $Q_{i+1}$ then $s_i\g_i=\bar{n}_{i,0}\b_0+\sum_{j=1}^{i-1}\bar{l}_{i,j}\g_j$ and $P_0^{\bar{n}_{i,0}}\prod_{j=1}^{i-1}T_j^{\bar{l}_{i,j}}$ is irreducible with respect to $\T_i$ since $\bar{l}_{i,j}<q'_j$ for all $j\le i-1$. Then by uniqueness of such a representation (Proposition \ref{uniqueness}) we have $n_{\bar{d},0}=\bar{n}_{i,0}$ and $l_{\bar{d},j}=\bar{l}_{i,j}$ for all $j\le i-1$. Then $T_i^{s_i}/(P_0^{n_{\bar{d},0}}\prod_{j=1}^{i-1}T_j^{l_{\bar{d},j}})=Q_i^{\bar{s}_i}/(P_0^{\bar{n}_{i,0}}\prod_{j=1}^{i-1}Q_j^{\bar{l}_{i,j}})$ and therefore, $\m_{\bar{d}}=\bar{\m}_i$. Thus
$$T_{i+1}=T_i^{s_i}-\m_{\bar{d}}P_0^{n_{\bar{d},0}}\prod_{j=1}^{i-1}T_j^{l_{\bar{d},j}}=Q_i^{\bar{s}_i}-\bar{\m}_iP_0^{\bar{n}_{i,0}}\prod_{j=1}^{i-1}Q_j^{\bar{l}_{i,j}}=Q_{i+1}.$$

In this example  $\{\b_i\}_{i\ge 0}\cup\{\b'_i\}_{i>0}$ is a set of generators for the value semigroup $\n(k[x,y,z]_{(x,y,z)})$.  It is also a minimal set of generators if $q_i> 1$ and $q'_i>1$ for all $i>0$. 
\end{example}

An example of this kind with $q_i=2$ and $q'_i=3$ for all $i>0$ has been considered in \cite{C-D-K}. Also, the sequence $\{\b_i\}_{i>0}\cup\{\bar{\g}_i\}_{i>0}$ of example \ref{Example1} satisfies the positivity condition of \cite{Mogh2}.

\begin{example}\label{Example2}
Let $\b_0=1$, $\b_1=1\frac{1}{2}$, $\b_2=3\frac{1}{4}$, $\b_3=6\frac{5}{8},\dots$ and $\bar{\g}_1=2\frac{1}{4}$, $\bar{\g}_2=3\frac{5}{8}$, $\bar{\g}_3=7\frac{1}{3}$,  $\bar{\g}_4=22\frac{1}{9},\dots$, where $\b_i=2\b_{i-1}+\frac{1}{2^i}$ and $\bar{\g}_{i+1}=3\bar{\g}_i+\frac{1}{3^{i-1}}$ for all $i>2$. Then for all $i>0$ in notation of section \ref{numerical data} we have 
$$
G_0=\ZZ,\quad G_i=\frac{1}{2^i}\ZZ\quad\quad\text{and}\quad \quad\bar{H}_1=\frac{9}{4}\ZZ, \quad\bar{H}_2=\frac{1}{8}\ZZ, \quad\bar{H}_{i+2}=\frac{1}{8\cdot3^i}\ZZ
$$
So $q_i=2$ for all $i>0$ and $\bar{s}_1=1$, $\bar{m}_1=2$, $\bar{s}_2=1$,  $\bar{m}_2=3$  and $\bar{s}_i=3$, $\bar{m}_i=3$ for all $i>2$. Notice that 
$\bar{\g}_1=-\b_0+\b_2$ and $\bar{\g}_2=-3\b_0+\b_3$, so that $\bar{r}_{1,0}=1$ and $\bar{r}_{2,0}=3$.  Also,
$2\b_1=3\b_0$ and $2\b_i=5\cdot2^{i-2}\b_0+\b_{i-1}$ for all $i>1$, and $3\bar{\g}_3=22\b_0$ and $3\bar{\g}_i=59
 \cdot3^{i-4}\b_0+\bar{\g}_{i-1}$ for all $i>3$. In particular, this shows $\bar{r}_{i,0}=0$  for all $i\ge 3$.  

Let $\l_i=\bar{\m}_i=1$ for all $i>0$. Since the inequalities $\b_{i+1}>q_i\b_i$ and $\bar{\g}_{i+1}>\bar{r}_{i,0}\b_0+\bar{s}_i\bar{\g}_i$ are satisfied for all $i>0$ the set of polynomials $\{P_i\}_{i\ge 0}\cup\{Q_i\}_{i>0}$ as constructed in section \ref{numerical data} defines a valuation $\n$ on $k(x,y,z)$.

We have $P_0=x$, $P_1=y$, $P_2=y^2-x^3$, $P_3=P_2^2-x^5y$ and  $Q_1=z$, $Q_2=xz-P_2$, $Q_3=x^3Q_2-P_3$, $Q_4=Q_3^3-x^{22 }$. The recursive formulas for $P_i$ and $Q_{i+1}$ when $i>3$ are
\begin{align*}
P_{i} &=P_{i-1}^2-x^{5\cdot 2^{i-3}}P_{i-2}\\
Q_{i+1}& =Q_i^3-x^{59\cdot 3^{i-4}}Q_{i-1}.
\end{align*}

We will directly compute several first members of the subsequence $\{T_i\}_{i>0}$ of jumping polynomials and identify the value semigroup generators with values less than 9. We use $(M,N)$-expansions to find the required values and residues. 

Since $T_1=z$ we get $\g_1=2\frac{1}{4}$, $s_1=1$, $m_1=2$, $\D_1=\{(1,0,0,1),(0,0,0,2),(0,0,1,1)\}$ and $\d_1=3$. Since   $xT_1=P_2+Q_2$ we have 
$$ T_1^2=x^3P_1+x^{-2}(P_3+2P_2Q_2+Q_2^2),\quad P_2T_1=x^4P_1+x^{-1}(P_3+P_2Q_2).$$ 
Thus the immediate successors of $T_1$ are
$$
\begin{array}{ll}
T_2 =xT_1-P_2=Q_2,\quad & \g_2=3\frac{5}{8}\\
T_3 =T_1^2-x^3P_1=x^{-2}(P_3+2P_2Q_2+Q_2^2),\quad & \g_3=4\frac{5}{8}\\
T_4 =P_2T_1-x^4P_1=x^{-1}(P_3+P_2Q_2),\quad &  \g_4=5\frac{5}{8}\\
\end{array}
$$

Consider $T_2$ and $\g_2=3\frac{5}{8}$. We have $s_2=1$, $m_2=3$, $r_{2,0}=3$, $\D_2=\{(3,0,0,0,0,1),\\
(0,0,0,0,0,2),(2,0,1,0,0,1),(0,0,0,1,0,1)\}$ and $\d_2=4$. Since  $x^3T_2=P_3+Q_3$ we get
$$T_2^2=x^4P_2+x^{-6}(P_4+2P_3Q_3+Q_3^2),\quad P_3T_2=x^7P_2+x^{-3}(P_4+P_3Q_3).$$
Also 
$$x^2P_2T_2=x^{-1}(P_2P_3+P_2Q_3) \text{ and } P_3T_1=x^{-1}(P_2P_3+P_3Q_2).$$
Thus the immediate successors of $T_2$ are
$$
\begin{array}{ll}
T_5 =x^3T_2-P_3=Q_3,\quad & \g_5=7\frac{1}{3}\\
T_6 =T_2^2-x^4P_2=x^{-6}(P_4+2P_3Q_3+Q_3^2),\quad & \g_6=7\frac{5}{16}\\
T_7 =x^2P_2T_2-P_3T_1=x^{-1}(-P_3Q_2+P_2Q_3),\quad &  \g_7=9\frac{1}{4}\\
T_8 =P_3T_2-x^7P_2=x^{-3}(P_4+P_3Q_3),\quad &  \g_8=10\frac{5}{16}\\
\end{array}
$$
It is helpful to notice that since $P_3=x^3T_2-T_5$ we can also write $T_8=x^3T_6-T_2T_5$.  From further computations it will follow that $x^3T_6$ and $T_2T_5$ are both irreducible with respect to $\T$, and therefore, $T_8$ is redundant by Lemma \ref{redundant_suf}. Also, the only successors of $T_8$ are $T_{k_1}=-T_2T_5$ and $T_{k_2}=0$, where $k_1=\sum_{j=0}^{8}\d_j$ and $k_2=\sum_{j=0}^{k_1}\d_j$. A similar argument applies to $T_7$. Since $P_2=xT_1-T_2$ and $T_2^2=x^4P_2+T_6$ we write $T_7=-x^6P_2-x^2T_6+T_1T_5$ to see that $T_7$ is redundant.

Consider $T_3$ and $\g_3=4\frac{5}{8}$. We have $s_3=1$, $m_3=3$, $\D_3=\{(0,0,0,0,0,0,1))\}$ and $\d_3=1$. Since  $x^2T_3=P_3+2P_2Q_2+Q_2^2$  and $x^3T_2=P_3+Q_3$ the only immediate successor of $T_3$ is
$$ T_9=T_3-xT_2=x^{-2}(2P_2Q_2+Q_2^2-Q_3), \quad\quad  \g_9=4\frac{7}{8}$$

Continuing in this manner we find  
$$
\begin{array}{ll}
 T_{10}=T_4-x^2T_2=x^{-1}(P_2Q_2-Q_3), \quad & \g_{10}=5\frac{7}{8}\\
  & \\
 T_{11}=T_5^3-x^{22}=Q_4, \quad &   \g_{11}=22\frac{1}{9}\\
  & \\
 T_{12}=x^6T_6-P_4=2P_3T_5+T_5^2, \quad & \g_{12}=13\frac{23}{24}\\
 T_{13}=T_6^2-x^8P_3=x^{-12}(P_5+hvt), \quad & \g_{13}=14\frac{21}{32}\\
 T_{14}=x^5P_2T_6-P_4T_1=-x^2P_3T_6+2P_3T_1T_5-x^2T_5T_6+T_1T_5^2, \quad & \g_{14}=15\frac{15}{16}\\
 T_{15}=x^3P_3T_6-P_4T_2=2x^7P_2T_5+x^3T_5T_6-T_2T_5^2, \quad & \g_{15}=17\frac{7}{12}\\
 T_{16}=P_4T_6-x^{14}P_3=x^6T_{13}-2P_3T_5T_6+T_5^2T_6, \quad & \g_{16}=13\frac{23}{24}\\
  & \\
  T_{17}=T_7+x^6P_2=-x^2T_6+T_1T_5, \quad &  \g_{17}=9\frac{5}{16}\\
   & \\
  T_{18}=T_8-x^3T_6=-T_2T_5, \quad &  \g_{18}=10\frac{23}{24}\\
 \end{array}
 $$
In particular, we see that $T_{12},T_{14},\dots, T_{18}$ are redundant and $T_{11}, T_{13}$ are of value greater than 9. It follows that the value semigroup generators of value less than 9 may only appear among values of the successors of $T_9$ and $T_{10}$. Moreover, only jumping polynomials of value less than 9 may be used to construct such successors of $T_9$ and $T_{10}$. From now on we will only keep track of successors of $T_9$ and $T_{10}$ with values less than 9.

Immediate successors of $T_9$ of value less than 9 may only have initial terms $xT_9$, $T_1T_9$, $P_2T_9$ or $T_2T_9$, since all other immediate successors of $T_9$ will have initial terms of value greater than 9.  $T_{10}$ has only one immediate successor  as $\g_{10}$ is a dependent value.
$$
\begin{array}{ll}
T_{9:1}=xT_9-2T_1T_2=-x^{-1}(Q_2^2+Q_3), \quad &  \g_{9:1}=6\frac{1}{4}\\
T_{9:2}=T_1T_9-2x^2P_1T_2=x^{-3}(2P_3Q_2+3P_2Q_2^2-P_2Q_3+Q_2^3-Q_2Q_3),  \quad &  \g_{9:2}=7\frac{1}{4}\\
T_{9:3}=P_2T_9-2P_1P_3=x^{-2}(2P_3Q_2+P_2Q_2^2-P_2Q_3+2x^2P_1Q_3),  \quad &  \g_{9:3}=8\frac{1}{4}\\
T_{9:4}=T_2T_9-2x^7P_1=x^{-2}(2P_2T_6+2x^4P_3+Q_2^3-Q_2Q_3),  \quad &  \g_{9:4}=8\frac{9}{16}\\
 & \\
 T_{10:1}=T_{10}-T_1T_2=-x^{-1}(Q_2^2+Q_3), \quad &  \g_{10:1}=6\frac{1}{4}\\
\end{array}
$$

Since $\g_{9:4}$ is an independent value all immediate successors of $T_{9:4}$ will have initial terms of value greater than 9. The immediate successors of $T_{9:1}, T_{9:2}, T_{9:3}, T_{10:1}$ are
$$
\begin{array}{ll}
T_{9:1:1}=T_{9:1}+x^3P_2=-x^{-1}(T_6+Q_3), \quad &  \g_{9:1:1}=6\frac{5}{16}\\
& \\
T_{9:2:1}=T_{9:2}-2x^4P_2=2T_6+x^{-3}(3P_2Q_2^2-P_2Q_3+Q_2^3-3Q_2Q_3),  \quad &  \g_{9:2:1}=7\frac{5}{16}\\
& \\
T_{9:3:1}=T_{9:3}-2x^5P_2=2xT_6+x^{-2}(P_2Q_2^2-P_2Q_3+2x^2P_1Q_3-2Q_2Q_3), &  \g_{9:3:1}=8\frac{5}{16}\\
& \\
T_{10:1:1}=T_{10:1}+x^3P_2=T_{9:1:1}, \quad  \text{ (redundant)}  & \g_{10:1:1}=6\frac{5}{16}\\
\end{array}
$$

Immediate successors of $T_{9:1:1}$ of value less than 9 may only have initial terms $xT_{9:1:1}$ or $T_1T_{9:1:1}$, while $T_{9:2:1}, T_{9:3:1}$ and $T_{10:1:1}$ each have one immediate successor.
$$
\begin{array}{ll}
T_{9:1:1:1}=xT_{9:1:1}+T_6=-T_5, \quad  \text{ (redundant) } & \g_{9:1:1:1}=7\frac{1}{3}\\
&\\
T_{9:1:1:2}=T_1T_{9:1:1}+\frac{1}{2}T_{9:4}=x^{-2}(-P_2Q_3+x^4P_3+x^4P_2Q_2-\frac{1}{2}Q_2^3-\frac{3}{2}Q_2Q_3),  &    \g_{9:1:1:2}=8\frac{7}{12}\\
& \\
T_{9:2:1:1}=T_{9:2:1}-2T_6=x^{-3}(3P_2Q_2^2-P_2Q_3+Q_2^3-3Q_2Q_3),  \, &  \g_{9:2:1:1}=7\frac{1}{2}\\
& \\
T_{9:3:1:1}=T_{9:3:1}-2xT_6=x^{-2}(P_2Q_2^2-P_2Q_3+2x^2P_1Q_3-2Q_2Q_3),  \, &  \g_{9:3:1:1}=8\frac{1}{2}\\
&\\
T_{10:1:1:1}=T_{10:1:1}-T_{9:1:1}=0 \quad & \\
\end{array}
$$

The only immediate successor of $T_{9:1:1:1}$ is 0. All immediate successors of $T_{9:1:1:2}$ are of value greater than 9. The immediate successors of $T_{9:2:1:1}$ and $T_{9:3:1:1}$  are
$$
\begin{array}{ll}
T_{9:2:1:1:1}=T_{9:2:1:1}-3x^6P_1=x^{-3}(3P_2T_6-P_2Q_3+3x^4P_3+Q_2^3-3Q_2Q_3),  &  \g_{9:2:1:1:1}=7\frac{9}{16}\\
& \\
T_{9:3:1:1:1}=T_{9:3:1:1}-x^7P_1=x^{-2}(P_2T_6-P_2Q_3+x^4P_3+2x^2P_1Q_3-2Q_2Q_3),  &  \g_{9:3:1:1:1}=8\frac{9}{16}\\
\end{array}
$$

The only immediate successors of $T_{9:2:1:1:1}$ of value less than 9 has initial term $xT_{9:2:1:1:1}$, while $T_{9:3:1:1:1}$  has only one immediate successor.
$$
\begin{array}{ll}

T_{9:2:1:1:1:1}=xT_{9:2:1:1:1}-\frac{3}{2}T_{9:4}=x^{-2}(-P_2Q_3-\frac{1}{2}Q_2^3-\frac{3}{2}Q_2Q_3),  &  \g_{9:2:1:1:1:1}=8\frac{7}{12}\\
& \\
T_{9:3:1:1:1:1}=T_{9:3:1:1:1}-\frac{1}{2}T_{9:4}=x^{-2}(-P_2Q_3+2x^2P_1Q_3-\frac{1}{2}Q_2^3-\frac{3}{2}Q_2Q_3), &  \g_{9:3:1:1:1:1}=8\frac{7}{12}\\
\end{array}
$$

Finally, the immediate successors of $T_{9:2:1:1:1:1}$ and $T_{9:3:1:1:1:1}$ are both redundant polynomials:
\begin{align*}
T_{9:2:1:1:1:1:1} =T_{9:2:1:1:1:1}-T_{9:1:1:2} &=-x^2P_3-x^2P_2Q_2\\
&=-x^2P_3-P_3T_1+x^6P_2+x^2T_6-T_1T_5; \\
T_{9:3:1:1:1:1:1} =T_{9:3:1:1:1:1}-T_{9:1:1:2} &=-x^2P_3+2P_1Q_3-x^2P_2Q_2\\
&=-x^2P_3+2P_1T_5-P_3T_1+x^6P_2+x^2T_6-T_1T_5. \\
\end{align*}
Thus the value semigroup generators of value less than 9 are $1, 1\frac{1}{2},  2\frac{1}{4},  3\frac{5}{8},  4\frac{7}{8}, 6\frac{5}{16}, 7\frac{1}{3}, 7\frac{9}{16}, 8\frac{7}{12}$.
\end{example}

\end{document}